\documentclass[11pt,leqno]{amsart}
\usepackage{amsmath,amsfonts,amsthm,amscd,amssymb,mathrsfs,graphicx,xcolor}
\usepackage{epstopdf}
\numberwithin{equation}{section}
\usepackage{fullpage}

\newcommand\s{\sigma}

\renewcommand{\L}{\Lambda}
\newcommand{\hal}{\hat{\l}}
\newcommand{\Hs}{\hat{\s}}
\newcommand{\tl}{\tilde{\l}}

\renewcommand\d{\partial}
\renewcommand\a{\alpha}
\renewcommand\b{\beta}

\def\g{\gamma}

\def\l{\lambda}
\def\eps{\varepsilon }
\def\e{\varepsilon}

\renewcommand\d{\partial}
\renewcommand\a{\alpha}

\renewcommand\b{\beta}

\newcommand\R{\mathbb R}

\def\g{\gamma}

\def\eps{\varepsilon}
\def\e{\varepsilon}

\def\l{\lambda}

\newcommand\br{\begin{remark}}
	\newcommand\er{\end{remark}}
\newcommand\bp{\begin{pmatrix}}
	\newcommand\ep{\end{pmatrix}}
\newcommand{\be}{\begin{equation}}
\newcommand{\ee}{\end{equation}}
\newcommand\ba{\begin{equation}\begin{aligned}}
\newcommand\ea{\end{aligned}\end{equation}}

\newcommand{\bap}{\begin{app}}
	\newcommand{\eap}{\end{app}}
\newcommand{\begs}{\begin{exams}}
	\newcommand{\eegs}{\end{exams}}
\newcommand{\beg}{\begin{example}}
	\newcommand{\eeg}{\end{exaplem}}
\newcommand{\bpr}{\begin{proposition}}
	\newcommand{\epr}{\end{proposition}}
\newcommand{\bt}{\begin{theorem}}
	\newcommand{\et}{\end{theorem}}
\newcommand{\bc}{\begin{corollary}}
	\newcommand{\ec}{\end{corollary}}
\newcommand{\bl}{\begin{lemma}}
	\newcommand{\el}{\end{lemma}}
\newcommand{\bd}{\begin{definition}}
	\newcommand{\ed}{\end{definition}}
\newcommand{\brs}{\begin{remarks}}
	\newcommand{\ers}{\end{remarks}}

\newtheorem{hypothesis}{Hypothesis}
\newtheorem{claim}{Claim}

\newcommand{\RR}{{\mathbb R}}

\newcommand{\NN}{{\mathbb N}}
\newcommand{\ZZ}{{\mathbb Z}}

\newcommand{\CC}{{\mathbb C}}

\newcommand{\sC}{\mathscr{C}}
\newcommand{\sN}{\mathscr{N}}
\newcommand{\sQ}{\mathscr{Q}}
\newcommand{\const}{\text{\rm constant}}

\newcommand{\sgn}{\text{\rm sgn}}
\newtheorem{theorem}{Theorem}[section]
\newtheorem{proposition}[theorem]{Proposition}
\newtheorem{corollary}[theorem]{Corollary}
\newtheorem{lemma}[theorem]{Lemma}

\theoremstyle{remark}
\newtheorem{remark}[theorem]{Remark}
\newtheorem{remarks}[theorem]{Remarks}
\theoremstyle{definition}
\newtheorem{definition}[theorem]{Definition}

\newtheorem{example}[theorem]{Example}

\newtheorem{obs}[theorem]{Observation}

\newcommand\cA{{\mathcal { A}}}

\newcommand\cV{{\mathcal  V}}

\newcommand\cC{{\mathcal  C}}

\newcommand\cK{{\mathcal  K}}
\newcommand\cL{{\mathcal  L}}
\newcommand\cN{{\mathcal  N}}
\newcommand\cE{{\mathcal  E}}

\newcommand\cQ{{\mathcal Q}}
\newcommand\cO{{\mathcal O}}

\newcommand\cS{{\mathcal S}}

\newcommand{\spec}{\operatorname{spec}}

\newcommand{\supp}{\text{\rm{supp}}}

\newcommand{\beq}{\begin{equation}}
\newcommand{\eeq}{\end{equation}}

\newcommand{\Hx}{\hat{x}}
\newcommand{\Ht}{\hat{t}}

\usepackage{color,soul}

\title{Diffusive stability of convective Turing patterns}
\author{Aric Wheeler}
\address{Indiana University, Bloomington, IN 47405}
\email{awheele@iu.edu }
\thanks{Research of A.W. was partially supported
under NSF grant no. DMS-1700279.}
\author{Kevin Zumbrun}
\address{Indiana University, Bloomington, IN 47405}
\email{kzumbrun@indiana.edu} 
\thanks{Research of K.Z. was partially supported
under NSF grants no. DMS-0300487 and DMS-0801745.}
\begin{document}
\begin{abstract}
Following the approach of \cite{E1,M1,M2,S1,S2,SZJV} for reaction diffusion systems,
we justify rigorously the Eckhaus stability criterion for stability of convective Turing patterns,
	as derived formally by complex Ginzburg-Landau approximation \cite{SS,NW,WZ}.
Notably, our analysis includes also higher-order, nonlocal, and even certain semilinear hyperbolic systems.
\end{abstract}

\maketitle
\section{Introduction}
In this paper, extending work of \cite{M1,M2,S1,S2,SZJV,WZ},
we validate by rigorous Lyapunov-Schmidt reduction 
the well-known formal Eckhaus stability criterion for general, convective, Turing patterns,
obtained by complex Ginzburg-Landau approximation \cite{E1,SS,NW,M3},
showing that Eckhaus stability is equivalent to the diffusive stability condition of Schneider,
a condition that is necessary and sufficient for linearized and nonlinear stability \cite{S1,S2,JZ,JNRZ1,SSSU}.

Following \cite{WZ,M3}, consider a family of perturbation equations in standard form 
\be\label{std}
u_t=L(\mu)u+\cN(u,\mu),
\ee
where $L(\mu)=\sum_{j=0}^m \cL_j(\mu)\d_x^j$ is a constant-coefficient differential operator
and $\cN$ is a general nonlinear functional of quadratic order in $u$ and $x$-derivatives,
under the following generalized {\it Turing assumptions} on the spectra of $L$ near the bifurcation point $\mu=0$,
or, equivalently, on the eigenvalues $\tilde \lambda_j(k,\mu)$ of the associated Fourier symbol
$S(k,\mu)=\sum_{j=0}^{m}\cL_j(\mu)(ik)^j$.

\begin{hypothesis}\label{hyp:Turing}
The symbol $S(k,\mu)$ and its eigenvalues $\{\tl(k,\mu),\tl_2(k,\mu),...,\tl_n(k,\mu) \}$ satisfy:\\

\noindent
	(H1) For $\mu<0$ and all $k\in\RR$, $\sigma(S(k,\mu))\subset\{z\in\CC:\Re z<0 \}$.\\
	(H2) For $\mu=0$ there is a unique $k_*>0$ such that $\Re\tl(k_*,0)=0$ and for $2\leq j\leq n$ $\Re\tl_j(k_*,0)<0$.\\
	(H3) For $\mu=0$ and all $k\not=\pm k_*$, we have that $\Re\tl(k,0)<0$ and for $2\leq j\leq n$ $\Re\tl_j(k,0)<0$.\\
	(H4) $\Re\d_\mu\tl(k_*,0)>0$, $\Re\d_k\tl(k_*,0)=0$ and $\Re\d_k^2\tl(k_*,0)<0$.
\end{hypothesis}

Under Hypotheses \ref{hyp:Turing}, fixing a wave number $\tilde k$ near $k_*$, there is a transcritical $SO(2)$ 
bifurcation from the constant solution to spatially-periodic traveling waves of period $\tilde k$ as $\mu$
increases near zero \cite{CK,CaK,M}.
Considerably more information, incorporating the continuum of $k$-dependent solutions,
is contained in the ``weakly unstable'' or ``weakly nonlinear'' approximation of Eckhaus \cite{E1}.

Let $r$ denote the eigenvalue of $S(k_*,0)$ associated with the critical eigenvalue $\tl(k_*,0)$,
so that (by complex conjugate symmetry, noting that $L$ is real-valued), $\tl(-k_*,0)=\overline{ \tl(k_*,0)}$, with
associated eigenvector $\bar r$.
Then, $u(x,t)=e^{i(k_*x + \Im \tl(k_*,0)t)}r+ c.c.$
is an exact nondecaying solution of the linearized equations 
$u_t=L(0)u$ at the bifurcation point $\mu=0$, where, here and elsewhere, $c.c.$ denotes complex conjugate.
Then, Eckhaus' ``weakly nonlinear'' expansion consists of formal asymptotic solutions of \eqref{std} of form
\ba\label{form}
U^{\e}(x,t)&=\frac{1}{2}\e A(\Hx,\Ht)e^{i\xi}r + \cO(\e^2) +c.c.,
\quad \xi=k_* \Big(x + \frac{\Im \tl(k_*,0)}{k_*}t \Big),\\
\mu&=\eps^2, \quad \Hx=\e(x+ \Im\d_k\tl(k_*,0)t), \quad  \Ht=\e^2t,
\ea
with amplitude $A\in \CC$ modulating the neutral linear solutions $e^{i(k_*x + \Im \tl(k_*,0)t)}r + c.c.$ at $\mu=0$.
Here, speeds $-\frac{\Im \tl(k_*,0)}{k_*}$ and $-\Im\d_k\tl(k_*,0)$ associated with moving frames $\xi$ and $\Hx$ 
may be recognized as phase and group velocities, respectively, of these underlying neutral linear solutions.

Substituting \eqref{form} into \eqref{std} yields \cite{M3,WZ} as a compatibility condition at $\cO(\eps^3)$
an {\it amplitude equation} consisting of the complex Ginzburg-Landau equation (cGL):
\begin{equation}\label{eq:cGL}
	A_{\Ht}=-\frac{1}{2}\d_k^2\tl(k_*,0)A_{\Hx\Hx}+\d_\mu\tl(k_*,0)A+\gamma|A|^2A,
\end{equation}
where the Landau constant $\gamma\in \CC$ is determined by the form of $\cN$
and spectral structure of $S(k_*,0)$.

The formal Ginzburg-Landau expansion \eqref{form}-\eqref{eq:cGL}, approximating behavior in neutral linear modes,
is expected to serve as an attractor for general small-amplitude solutions of \eqref{std}, with all
other linear modes strictly exponentially decaying.  For results on finite ($\cO(\eps^{-2})$) time
approximation of solutions of \eqref{std} by solutions of (cGL) see \cite{E2,vH,M3} and references therein.

\subsection{Existence}\label{s:existence}
Under the supercriticality condition $\Re \gamma \Re \d_\mu\tl(k_*,0)<0$, 
there exist periodic solutions 
\be\label{GLper}
A=e^{i(\kappa \Hx+ \omega \Ht)} \alpha,  \quad \alpha\equiv \const,
\quad
i\omega= -\frac12 \d_k^2\tl(k_*,0) \kappa^2 +\d_\mu\tl(k_*,0)+\gamma|\alpha|^2,
\ee
of (cGL) corresponding through \eqref{form}--\eqref{eq:cGL} to expected bifurcating traveling-wave solutions
\ba \label{expected}
U^{\e}(x,t)&=\frac{1}{2}\e \alpha  e^{i(kx + \Omega t)}r + \cO(\e^2) +c.c.
\quad
k=k_*+\eps \kappa, \quad
\Omega= \Im \tl(k_*,0) + \eps\kappa \partial_k\tl(k_*,0) + \eps^2 \omega ,
\ea
for wave numbers in the range
\be\label{krange}
\kappa^2 <\kappa_E^2:= 2 \Re \d_\mu\tl(k_*,0)/ \d_k^2\Re \tl(k_*,0),
\ee
where
\ba\label{alphaomega}
|\alpha|&= \sqrt{ \Re \gamma^{-1}\Big( \frac12 \d_k^2\Re \tl(k_*,0) \kappa^2 -\Re \d_\mu\tl(k_*,0) \Big)},\\
\omega&= -\frac12 \Im \d_k^2\tl(k_*,0) \kappa^2 +\Im \d_\mu\tl(k_*,0)+\Im \gamma|\alpha|^2.
\ea

The following result established by Lyapunov-Schmidt reduction in \cite{WZ} shows that there indeed 
exist exact traveling-waves solutions of \eqref{std} near the predicted solutions \eqref{expected}, 
bifurcating from the constant solution $u\equiv 0$ as $\mu$ increases near zero.

\begin{proposition}[\cite{WZ}]\label{mainLS}
Under Turing Hypotheses \ref{hyp:Turing},
for quasilinear nonlinearity $\cN$ and $\mu=\eps^2$, for any $\nu_0>0$ there exists 
$\eps_0$ such that for $\eps\in[0,\eps_0)$ 
and $\kappa^2 \leq (1-\nu_0) 2 \Re \d_\mu\tl(k_*,0)/ \d_k^2\Re \tl(k_*,0)$ there exists
	a unique (up to translation, i.e., up to choice of $\alpha$) 
	small spatially periodic traveling-wave solution $\bar U^\eps (kx+\bar \Omega t)\not \equiv 0$ 
of \eqref{std}, $\bar U$ $2\pi$-periodic, with $k=k_*+ \eps \kappa$, satisfying
\ba\label{soln}
	\bar U^{\e}(z)&= \Big(\frac{1}{2}\e \alpha  e^{iz}r + c.c.\Big) +\cO(\e^2),\\
	\bar \Omega &= \Big( \Im \tl(k_*,0) + \eps\kappa \partial_k\tl(k_*,0) + \eps^2 \omega \Big) + O(\eps^3),
\ea
where $\alpha\in \CC$ and $\omega\in \RR$ satisfy \eqref{alphaomega},
while for $\eps\in[0,\eps_0)$ and
$\kappa^2 \geq (1+\nu_0) 2 \Re \d_\mu\tl(k_*,0)/ \d_k^2\Re \tl(k_*,0)$ there exist no such small nontrivial solutions.
In the ($O(2)$-symmetric) generalized reaction diffusion case that $L$ and $N$ depend only 
on even derivatives or even powers of odd derivatives of $u$, $\bar \Omega\equiv 0$ and $\bar U^\eps$ is even for $\alpha\in \RR$.
\end{proposition}

The $O(2)$-symmetric case is the classic stationary Turing bifurction of \cite{T}.
The $SO(2)$-symmetric case, for which in general $\Omega\neq 0$,
represents the ``convective'' Turing bifurcation of the title \cite{WZ}.

\subsection{Stability}\label{s:stability}
As described in Section \ref{sec:cglstability} 
the linearized stability problem for periodic solutions \eqref{GLper} of (cGL) may be explicitly solved 
in terms of exponential functions, yielding the {\it Eckhaus stability criterion} \cite{AK,E1,M3}:
\be\label{Estab}
\kappa^2<\kappa_S^2:= 2 \frac{ \Im(\partial_k \tl(k_*,0) )\Im \gamma \Re (\partial_\mu \tl(k_*,0) ) \Re \gamma 
+\Re(\partial_k \tl(k_*,0)) \Re (\partial_\mu \tl(k_*,0)) \Re \gamma^2}
{\Re \partial_k \tl(k_*,0)(2\Im \gamma^2\Re(\partial_k \tl(k_*,0)) +\Im(\partial_k \tl(k_*,0)) 
\Im(\gamma) \Re(\gamma)+ 3\Re(\partial_k \tl(k_*,0)) \Re( \gamma^2))},
\ee
where $\kappa_S^2<\kappa_E^2$; see \eqref{eq:genericCGL}, \eqref{eq:StableBand}, \eqref{ineq}), and Remark \ref{normrmk}.
As an exact stability condition for the approximating (cGL) solution \eqref{GLper} within the formal
attractor given by the set of all (cGL) solutions, this serves as a {\it formal stability criterion} for
bifurcating waves $U^\eps$ as solutions of \eqref{std}.

An exact stability criterion for spatially periodic waves $U^\eps$ as solutions of \eqref{std}, meanwhile, 
is given by the {\it diffusive stability condition} of Schneider \cite{S1,S2}, which we now describe.
Let $\cL^{\eps,\kappa}=\cS(\partial_x,x;\eps,\kappa)$ 
denote the linearized operator about solution $U^{\eps,\kappa}$, expressed in
a co-moving coordinate frame for which $\bar U^{\eps,\kappa}$ is stationary.
The associated periodic-coefficient Floquet operator
\be\label{stdBloch}
\cL^{\eps,\kappa}_\sigma:= \cS(\partial_x + i\sigma, x;\eps,\kappa),
\ee
has the property \cite{G} that the spectrum of $\cL^{\eps,\kappa}$, considered as an operator on the whole line,
is given by the union over $\sigma \in [-\pi/X^\eps, \pi/X^\eps)$ of spectra of $\cL^{\eps,\kappa}_\sigma$,
considered as operators on $[0,X^{\eps,\kappa})$ with periodic boundary conditions, where 
$X^{\eps,\kappa}= 2\pi/k:= 2\pi/(k_*+\kappa \eps)$ is the period of $U^{\eps,\kappa}$.

An evident necessary condition for linearized stability of $U^{\eps,\kappa}$ is thus
\be\label{specstab}
\hbox{\rm $\Re \spec( \cL^{\eps,\kappa}_\sigma) \leq 0$ for all $\sigma \in [-\pi/X^\eps, \pi/X^\eps)$.}
\ee
The sufficient condition for time-exponential stability 
$\Re \spec( \cL^{\eps,\kappa}_\sigma) <0$ for all $\sigma \in [-\pi/X^\eps, \pi/X^\eps)$ is in this
case not possible, since $\partial_x \bar U^{\eps,\kappa}$ by translation-invariance of the underlying
equation \eqref{std} is a zero eigenfunction of $L^{\eps,\kappa}_0$.
However, as shown in \cite{S1,S2,JZ,SSSU,JNRZ1}, 
assuming the transversality condition that $0$ be a simple eigenvalue of $\cL^{\eps,\kappa}_0$,
a sufficient condition for {\it time-algebraic} linear and
nonlinear stability is Schneider's {\it diffusive stability condition}:
\be\label{diffstab}
\hbox{$\Re \spec (\cL^{\eps,\kappa}_\sigma) \leq -\theta \sigma^2$ 
for all $\sigma \in [-\pi/X^\eps, \pi/X^\eps)$, for some $\theta>0$.}
\ee

Our following, first main result rigorously validates the formal Eckhaus stability condition.

\begin{theorem}[Stability]\label{main}
Under Turing Hypotheses \ref{hyp:Turing},
for quasilinear nonlinearity $\cN$ and $\mu=\eps^2$, for any $\nu_0>0$ there is
$\eps_0>0$ such that, for $\eps\in[0,\eps_0)$ and $\kappa^2 \leq (1-\nu_0)\kappa_S^2$, 
	the solutions $\bar U^{\eps,\kappa}$ of \eqref{std} described in Proposition \ref{mainLS}
	satisfy \eqref{diffstab} hence are linearly and nonlinearly stable, 
	while for $\eps\in[0,\eps_0)$ and $\kappa^2 \geq (1+\nu_0)\kappa_S^2 $, they fail \eqref{specstab}
	hence are linearly exponentially unstable.
\end{theorem}

\subsection{Behavior/spectral expansion}\label{s:behavior}
Similarly, from the heuristic picture of \eqref{form}--\eqref{eq:cGL} as an approximate
attracting manifold for \eqref{std},
we may expect that asymptotic behavior of perturbed stable periodic solutions $U^{\eps,\kappa}$ be well
described by asymptotic behavior of solutions of (cGL).

As shown in \cite{JNRZ2,SSSU}, time-asymptotic behavior is closely related to the second-order expansion 
\be\label{Texp}
\lambda_*(\sigma)= \alpha i\sigma - \beta \sigma^2 +o(\sigma^2)
\ee
of the ``neutral,'' or ``critical'' spectral curve $\lambda=\lambda_*(\sigma)$ bifurcating from 
the simple translational zero eigenvalue $\lambda=0$ of $\cL^{\eps,\kappa}_0$, and
so we expect agreement here as well.
As shown in Section \ref{sec:cglstability}, the spectrum of the Floquet operators $\tilde{ \cL}^{\kappa}_\sigma$
about solutions \eqref{GLper} of (cGL) consists 
of a pair of eigenvalues
\be\label{cGLevals}
\tilde \lambda_j(\sigma)=\tilde c^j_0 + \tilde c^j_1  \sigma + \tilde  c^j_2 \sigma^2 , \qquad j=1,2
\ee
(see \eqref{cGLlambdas}) 
of which $\tilde \lambda_2$ is the critical eigenvalue of solution \eqref{GLper} 
passing through $\lambda=0$ for $\sigma$ ($\ell$ in the notation of Section \ref{sec:cglstability}) equal to $0$
and $\lambda_1$ a small negative eigenvalue bifurcating from the additional zero eigenvalue at
$(\eps, \sigma)=(0,0)$, determining the rate of attraction toward the approximate center manifold given by
the complex Ginzburg-Landau approximation \eqref{form}--\eqref{eq:cGL}.
More precisely, $\Re \tilde c^1_0<0$ in the supercritical case $\Re \gamma<0$, while
$\tilde c^2_0=0$, $\Re \tilde c^2_1=0$, and $\Re \tilde c^2_2<0$ in the stable case $\kappa^2< \kappa_S^2$.

Our following, second main result shows that Eckhaus' weakly nonlinear formalism successfully
predicts the expansions of the corresponding pair of critical eigenvalues of $\cL^{\eps,\kappa}_\sigma$
bifurcating from the double root at $(\eps,\sigma)=(0,0)$, and thus time-asymptotic behavior of perturbations
of $U^{\eps,\kappa}$.

\begin{theorem}[Spectral expansion]\label{main2}
	Let $u_{\eps,\kappa}$ be the solution from Proposition \ref{mainLS}.  Then there exist
	 $\tilde\eps_0\in(0,\eps_0]$ ($\eps_0$ as in Proposition \ref{mainLS}), $\sigma_0>0$ and $\delta>0$ 
	 such that for all $\eps\in[0,\tilde\eps_0)$, all $\sigma\in[0,\sigma_0)$  and all $\kappa^2 \leq \kappa_E^2$, 
	 the spectrum of $\cL^{\eps,\kappa}_\sigma$ has the decomposition:
	\be
	\begin{split}
		\spec( \cL^{\eps,\kappa}_\sigma)=S\cup\{\lambda_1,\lambda_2\}.
	\end{split}
	\ee
	where $\Re\lambda<-\delta$ for $\lambda\in S$ and $|\lambda_j|<<1$.
Moreover, setting $\sigma =:\eps \hat \sigma$, $\lambda_j=:\eps^2 \hat \lambda_j$ in accordance with the
	Ginzburg-Landau scaling \eqref{form}(ii), we have
	 \be\label{lamexp}
	\begin{split}
		&	\hat \lambda_{1}= \hat c^1_0+ \mathcal O( \hat \sigma),\\
		&	\hat \lambda_{2} =  \hat c^2_1 \hat \sigma
		+ \hat  c^2_2 \hat \sigma^2+\mathcal O(\hat \sigma^3),
	\end{split}
	\ee
	where
	\be\label{connect}
	\hbox{\rm $\hat c^1_0- \tilde c^1_0=\cO(\eps)$,
	$\hat c^2_2- \tilde c^2_2=\cO(\eps)$,
	and	
	$\hat c^2_1- \tilde c^2_1= i\tilde d + \cO(\eps)$ with $\tilde d$ real.}
	\ee
\end{theorem}

We note that the discrepancy $i \tilde d$  in \eqref{connect} (computed explicitly in Section 
\ref{sec:gencase})
corresponds to an extraneous term $\tilde d\partial_x$ in the linearized equations for (cGL), due to
the fact that (cGL) is in general posed in a frame that is moving with respect to the co-moving frame in which
$U^{\eps,\kappa}$ is stationary. 
As a purely imaginary term, it does not affect stability, but does affect behavior via convection at rate $\tilde d$.
This affine shift does not occur in the $\cO(2)$ symmetric generalized reaction diffusion case treated in
\cite{M1,M2,S1,S2,SZJV,S}, for which the wave and various coordinate frames are all stationary.
It is one of the main new subtleties in the analysis of the general $SO(2)$ case.

\subsection{Discussion and open problems}\label{s:disc}
Restricted to the reaction diffusion case, main Theorems \ref{main}-\ref{main2} recover the results 
obtained in \cite{S1,SZJV} for the Swift-Hohenberg and Brusselator models.
The extension from these individual models to general reaction diffusion systems, 
though expected, is new, resolving an important open problem cited in \cite{SZJV}.
In the general, convective case, Theorems \ref{main}-\ref{main2} are to our knowledge the first 
such results obtained for any system.

As hinted, perhaps, by the discussion at the beginning of Section \ref{s:behavior},
Proposition \ref{mainLS} and Theorem \ref{main2} together show in fact that time-asymptotic behavior
of stable bifurcating traveling waves of \eqref{std} is well predicted by that of periodic solutions 
\eqref{GLper} of (cGL).
For, as shown in \cite{JNRZ2,SSSU}, time-asymptotic behavior for either equation (exact or approximate)
is determined by the formal second-order {\it Whitham expansions} \cite{W,HK} determined by
the nonlinear dispersion relation for the associated existence problems:
$k_t + \Omega(k)_x= (d(k)k_{x})_x$ and $\kappa_t+\omega(k)_x=(\tilde d( \kappa)_{x})_x$, respectively;
specifially, the values of $d$ and $\tilde d$
and second-order Taylor expansions of $\Omega$ and $\omega$ at $k_*+ \eps \kappa$ and $\kappa$. 
The values of $d$, $\tilde d$ and the first-order Taylor expansions of $\Omega$, $\omega$, moreover, are
determined by the spectral expansion of the neutral curve \eqref{Texp}
given in Theorem \ref{main2}, hence agree after rescaling up to $O(\eps)$ error.
On the other hand differentiating both sides of \eqref{soln} in Proposition \ref{mainLS}, 
we find that $\partial_{\kappa}^2 \Omega= \eps^2 \partial_{\kappa}^2\omega + O(\eps^3)$,
hence, after (cGL) rescaling, again agree up to $O(\eps)$ error.


For readability, Theorems \ref{main}-\ref{main2} are established first in Sections \ref{sec:coper}-\ref{sec:gencase}
for the simplest type of nonlinearity $\cN(u)(x)= N(u(x))$ depending on $u$ but not its derivatives. 
More specifically, Theorem \ref{main} is obtained from Theorem \ref{thm:genstab} in this case, 
and Theorem \ref{main2} is obtained from Theorem \ref{thm:specagree}.
The extension to general nonlinearities is given in Section \ref{sec:other}.
There, the key result is Theorem \ref{thm:genmatch}, from which
Theorems \ref{main} and \ref{main2} can be obtained in the same way as Theorems \ref{thm:specagree} and \ref{thm:genstab} are obtained from Theorem \ref{thm:ReducedEqn}.
Notably, this includes
not only the quasilinear case described in the theorems, but also a wide class of nonlocal models, as arise
in water waves \cite{L}, chemotaxis, etc.
As described in \cite[Rmk. 4.4]{WZ}, one may treat by the same methods also 
``nonresonant'' semilinear hyperbolic models, an extension that may be useful for applications to kinetic models. See Remark \ref{rem:nonlocal} for a sketch of the extension to the nonlocal case.

As in the existence problem \cite{WZ}, a new technical difficulty arising in the
convective case is the presence of multiple moving coordinate frames, with periodic (cGL) waves moving at different
speed than corresponding bifurcating traveling-wave solutions of \eqref{std}: 
specifically, speeds $-\omega/\kappa$ vs. $\Omega/k$ in \eqref{soln}.
For the stability problem, there is yet a third relevant coordinate frame, further complicating the analysis.
Namely, as described in Section \ref{sec:cglstability}, the linear stability analysis of periodic waves \eqref{GLper}
of (cGL) is most conveniently carried out by reduction to constant coefficients, enforcing a coordinate
frame ($\xi$ \eqref{form}) moving with speed $-\Im \tl(k_*,0)/k_*$ different from either $-\omega/\kappa$ 
or $\Omega/k$.

Finally, we mention an important class of mechanochemical/hydrodynamical bifurcations arising, e.g., in vasculogenesis
 \cite{MO,Ma,Mai,SBP}, to which our assumptions {\it do not apply}, namely, systems
$\d_t w+\d_x f(w)=g(w)+ \d_x(b(w)\d_x w)$, $w\in \R^n$ for which $g$ is of partial rank $r$, 
satisfying $\ell_j g\equiv 0$ for constant vectors $\ell_j$, $j=1, \dots, n-r$.
As described further in \cite{WZ}, these possess {\it conservation laws}
$ \int \ell_j w \, dx\equiv \const$, $j=1, \dots n-r$,
as a consequence of which Hypothesis (H3) necessarily fails at $k=0$.
This case has been treated in \cite{MC,S} for a model $O(2)$-invariant Swift-Hohenberg type equation
possessing a single conservation law, and in \cite{HSZ} for $SO(2)$ models of
B\'enard-Marangoni and thin-film flow,
with the result that behavior is well-predicted by an extended 
Ginzburg-Landau approximation consisting of a real Ginzburg-Landau equation in $A$, coupled with
a scalar diffusion equation in $B$, where $A$ as here is amplitude of critical linear modes and $B$
is related to the conserved quantity induced by the conservation law.

The extension of this analysis to the general convective, multi-conservation law case relevant to vasculogenesis
we consider an important open problem.
An important further extension would be 
to treat the case of incomplete parabolicity $\det B=0$ occurring for a number of physical 
models.


\section{Linear Stability of periodic Complex Ginzburg-Landau solutions}\label{sec:cglstability}
We begin by recalling the linearized stability analysis for periodic (cGL) solutions \cite{AK,M3}.
	Consider the general complex Ginzburg-Landau equation
	\begin{equation}\label{eq:genericCGL}
		A_t=aA_{xx}+bA+c|A|^2A
	\end{equation}
	where $a,b,c$ are complex numbers with the appropriate signs on their real parts, i.e. $\Re(a),\Re(b)>0$ and $\Re(c)<0$. Assume that $A$ has the form
	\begin{equation}\label{eq:genericAnsatz}
		A(x,t)=\a e^{i(\kappa x-\omega t)}
	\end{equation}
	where, without loss of generality, $\kappa ,\omega\in\RR$ are yet to be determined constants and $\a$ is positive real.
	This can be accomplished by performing a phase shift and taking advantage of the $SO(2)$ invariance of \eqref{eq:genericCGL}. 
	
	Plugging \eqref{eq:genericAnsatz} into \eqref{eq:genericCGL}, we obtain the nonlinear
	dispersion relation 
	\begin{equation}\label{eq:qOmegaEqn}
		-i\omega=-a\kappa ^2+b+c\a^2.
	\end{equation}
	Solving real and imaginary parts separately in \eqref{eq:qOmegaEqn}, we find that
		$$
			\omega=\Im(a)\kappa^2-\Im(b)-\Im(c)\a^2,
			\quad
			0=-\Re(a)\kappa^2+\Re(b)+\Re(c)\a^2.
			$$
	The second equation is solvable on the range of existence
	\begin{equation}\label{eq:CGLExistence}
		\kappa^2\leq \kappa^2_E:=\frac{\Re b}{\Re a}
	\end{equation}
	yielding $\omega=\Im(a)\kappa^2-\Im(b)-\Im(c)\a^2$ and
		$\a^2=\frac{-\Re b+\Re a \kappa^2}{\Re c} $ as functions of $\kappa$.  

		\subsection{Linear stability analysis}\label{s:linstab_cGL}
	We now perturb the solution $A(x,t)$ constructed above by
	\begin{equation}
		u(x,t)=\left(\a+B(x,t) \right)e^{i(\kappa x-\omega t)},
	\end{equation}
	factoring out periodic behavior to obtain $B$ as a perturbation of a constant solution $\alpha$.

	Plugging this Ansatz into \eqref{eq:genericCGL} gives
	\begin{equation}\label{prestep}
		-i\omega u+B_te^{i(\kappa x-\omega t)}=-\kappa ^2u+2i\kappa B_xe^{i(\kappa x-\omega t)}+bu+c(|\a|^2\a+2|\a|^2B+\a^2\overline{B}+\cO(|B|^2))e^{i(\kappa x-\omega t)}.
	\end{equation}

	We can simplify \eqref{prestep} using the fact that $\a e^{i(\kappa x-\omega t)}$ is a solution, factoring out the exponential term $e^{i(\kappa x-\omega t)}$, and dropping the $\cO(|B|^2)$ terms, to obtain
	a constant-coefficient equation
	\begin{equation}
		-i\omega B+B_t=aB_{xx}+2i\kappa aB_x-\kappa^2aB+bB+c\a^2B+c\a^2(B+\overline{B})
	\end{equation}
	in the modified unknown $B$.
	Applying \eqref{eq:qOmegaEqn} gives, finally, the linearized (cGL) equation
	\begin{equation}\label{eq:LinearizedGenericCGL}
		B_t=aB_{xx}+2i\kappa aB_x+c\a^2(B+\overline{B}),
	\end{equation}
	again, in the coordinates with background periodic behavior factored out.

	We now write $B=u+iv$ where $u$ and $v$ are two real-valued functions. This gives the system
	\begin{equation}\label{eq:LinearizedSystem}
		\bp u_t \\ v_t \ep=\bp \Re a & -\Im a \\ \Im a & \Re a\ep\bp u_{xx} \\ v_{xx} \ep+\bp -2\kappa\Im a & -2\kappa\Re a \\ 2\kappa\Re a & -2\kappa\Im a\ep \bp u_x \\ v_x \ep +\bp 2\a^2\Re c & 0 \\ 2\a^2\Im c & 0\ep\bp u \\ v \ep.
	\end{equation}
	Assume that
	\begin{equation}
		\bp u \\ v\ep=\bp u_0\\ v_0\ep e^{i\sigma x+\l t}
	\end{equation}
	for $\sigma \in\RR$ small and $\l\in\CC$ to be determined. Plugging this 
	into \eqref{eq:LinearizedGenericCGL} gives
	\begin{equation}\label{eq:hopefulreducedeqn}
		\l\bp u_0\\v_0\ep =\left(-\sigma ^2\bp \Re a & -\Im a \\ \Im a & \Re a\ep+i\sigma \bp -2\kappa\Im a & -2\kappa\Re a \\ 2\kappa\Re a & -2\kappa\Im a\ep+\bp 2\a^2\Re c & 0 \\ 2\a^2\Im c & 0\ep\right)\bp u_0\\v_0\ep,
	\end{equation}
	i.e., that $\l$ is an eigenvalue of the matrix on the right-hand side. 

	Computing the eigenvalues of this matrix and then Taylor expanding them about $\sigma=0$ gives
	\begin{equation}\label{cGLlambdas}
		\begin{split}
			\l_1=2\a^2\Re c+\cO(\sigma )\\
			\l_2=-2i\kappa\left(\Im a -\frac{\Im c \Re a}{\Re c}\right)\sigma +\frac{(2\kappa^2\Im c ^2\Re a^2+\a^2\Im a\Im c\Re c^2+\Re a\Re c^2(2\kappa^2\Re a+\a^2\Re c))\sigma^2}{(\a^2\Re c^3)}+\cO(\sigma^3)
		\end{split}
	\end{equation}
	Requiring that both eigenvalues have negative real part yield the range of stability
	\begin{equation}\label{eq:StableBand}
\kappa^2<\kappa_S^2:=\frac{\Im a\Im c\Re b\Re c+\Re a\Re b\Re c^2}{\Re a(2\Im c^2\Re a+\Im a\Im c\Re c+3\Re a\Re c^2)},
	\end{equation}
	corresponding to the Eckhaus stability criterion \cite{E1,AK}.

The band of stability \eqref{eq:StableBand} is nontrivial under the {\it Benjamin-Feir-Newell} criterion
	\begin{equation}\label{Eq:BFNCriterion}
			\Im a\Im c\Re b\Re c+\Re a\Re b\Re c^2>0.
	\end{equation}

\begin{remark}\label{normrmk}
One can normalize $a=1+i\tilde{\a}$, $b=1$, $c=-1-i\tilde{\beta}$ via a sequence of coordinate changes to 
recover the more usual forms of the amplitude $\a^2=1-\kappa^2$, existece and stability boundaries 
$\kappa_E^2=1$ and $\kappa_S^2=\frac{1+\tilde{\a}{\tilde{\b}}}{3+\tilde{\a}\tilde{\b}+2\tilde{\b}^2}$,
	and Benjamin-Feir-Newell criterion $1+\tilde{\a}\tilde{\b}>0$ \cite{SD},
from which we readily see also the relation 
	\be\label{ineq}
	|\kappa_S|<|\kappa_E|.
	\ee
	\end{remark}

		
	\section{Co-periodic Stability}\label{sec:coper}
	We next carry out a rigorous stability analysis by Lyapunov-Schmidt reduction
	for the bifurcating periodic solutions of \eqref{std},
	starting here with the co-periodic case, or stability with respect to perturbations that are periodic
	with the same period as the background wave.
	We analyze stability with respect to general perturbations in the following section.
	For simplicity of exposition, we restrict in both this and the next section to the case of
	the simplest possible nonlinearity $\cN$, consisting of a function on $u$ alone, with no derivatives,
	treating general nonlinearities in Section \ref{sec:other}.

	Define the co-periodic Bloch operator $B(\e,\kappa,\l)$ by
	\begin{equation}\label{eq:BlochOperator}
		B(\e,\kappa,\l)=L(k,\mu)+d(\e,\kappa)k\d_\xi+D\cN(\tilde{u}_{\e,\kappa})-\l
	\end{equation}
	where $L$ and $\cN$ are as in \eqref{std}, with $\cN(u)= N(u(x))$ 
	quadratic order $C^s$ function. 
	The zero set of $ B(\e,\kappa,0)$ is the spectrum of $\cL^{\eps,\kappa}$ considered as an operator
	on the interval $[0,X^{\eps,\kappa})$ with periodic boundary conditions, where $\cL^{\eps,\kappa}$ as
	in the introduction is the linearization of \eqref{std} about $U^{\eps,\kappa}$, and 
	$X^{\eps,\kappa}=1/k$ is the period of $U^{\eps,\kappa}$, with $k=k_*+ \eps\kappa$.

	For later use, we compute $B(0,\kappa,\l)$ and the first two $\e$-derivatives of $B(\e,\kappa,\l)$ evaluated at $\e=0$:
	\begin{equation}\label{eq:D0BDeps0}
		B(0,\kappa,\l)=L(k_*,0)+k_*d_*\d_\xi-\l
	\end{equation}
	\begin{equation}\label{eq:DBDeps}
		B_\e(0,\kappa,\l)=\kappa L_k(k_*,0)D_\xi+\left(k_*d_\e(0,\kappa)+\kappa d_*\right)\d_\xi+D^2\cN(0)(\d_\e\tilde{u}_{0,\kappa},\cdot)
	\end{equation}
	\begin{equation}\label{eq:D2BDeps2}
	\begin{split}
		B_{\e\e}(0,\kappa,\l)=\kappa^2L_{kk}(k_*,0)D_\xi^2+\mu_{\e\e}L_\mu(k_*,0)+\left(d_{\e\e}(0,\kappa)k_*+2\kappa d_\e(0,\kappa) \right)\d_\xi+\\+D^2\cN(0)(\d_\e^2\tilde{u}_{0,\kappa},\cdot)+D^3\cN(0)(\d_\e\tilde{u}_{0,\kappa},\d_\e\tilde{u}_{0,\kappa},\cdot),
		\end{split}
	\end{equation}
	where $L_k$, $L_{kk}$, $L_\mu$ and $D_\xi$ are the operators
	\begin{align}\label{eq:derivop1}
		L_k(k,\mu)U(\xi)&=\sum_{\eta\in\ZZ}S_k(k\eta,\mu)\hat{U}(\eta) e^{i\eta\xi},\\
		L_{kk}(k,\mu)U(\xi)&=\sum_{\eta\in\ZZ}S_{kk}(k\eta,\mu)\hat{U}(\eta)e^{i\eta\xi},\\
		L_\mu(k,\mu)U(\xi)&=\sum_{\eta\in\ZZ}S_\mu(k\eta,\mu)\hat{U}(\eta)e^{i\eta\xi},\\
		D_\xi{U}(\eta)&=\sum_{\eta\in\ZZ}\eta\hat{U}(\eta)e^{i\eta\xi}.
	\end{align}

		A few remarks are in order about operators \eqref{eq:derivop1}.
		First note that neither $L_k$ nor $D_\xi$ are real operators; but their composition is a real operator. The second remark, is that one needs bounds on the derivatives of the symbol with respect to $k$ to properly make sense of these formulae. Since we will only be interested in the $\pm 1$ modes of these expansions; the behavior of the symbol and it's derivatives near infinity is irrelevant for us here.

	As in \cite{WZ}, we set $P$ be the projection onto the kernel of $L(k_*,0)+k_*d_*\d_\xi$. To be precise, we recall that $P$ is given by
	$$
		PU(\xi)=r\ell\hat{U}(1)e^{i\xi}+c.c.,
	$$
	where $\ell$ and $r$ are the left and right (resp.) eigenvectors associated to the critical eigenvalue $\tl(k_*,0)$.
	Let $W$ be in the kernel of $B(\e,\kappa,\l)$, and write $W=\Upsilon_{\b}+\cV(\e,\kappa,\l)$ 
	for $\b\in\CC$ and $P\cV=0$, where
	\begin{equation}\label{eq:upsilondef}
		\Upsilon_{\b}=\Re\left(\b e^{i\xi}r\right). 
	\end{equation}
	\begin{remark}
If $r$ is a real vector, as in the reaction diffusion case, and $\b=\b_1+i\b_2$, then
		\begin{equation}
			\Upsilon_\b=\Re\left(\b e^{i\xi} \right)r=\b_1\cos(\xi)r-\b_2\sin(\xi)r.
		\end{equation}
	\end{remark}

	There are many possible forms for $\Upsilon_\b$ because $\ker(L(k_*,0)+d_*k_*\d_\xi)$ is a 2-dimensional 
	real vector space. Ultimately all are equivalent, however, this particular form in \eqref{eq:upsilondef} makes it slightly easier to identify the action of the Fourier multiplier operators. Another reason why this way is slightly preferable is that we may identify $B(x,t)$ in the derivation as $\b e^{\l t}$. That said, some caution must be made as this suggests that the formulas in what follows will depend analytically on $\b$, which will not be the case.

	 We begin solving $B(\e,\kappa,\l)W=0$ by first applying $(I-P)$ to both sides.
	\begin{equation}
		(I-P)B(\e,\kappa,\l)(\Upsilon_\b+\cV)=0
	\end{equation}
	But $(I-P)B(\e,\kappa,\l)$ is invertible for $\e=0$ and $\l$ small. So we may apply the inverse function theorem to solve for $\cV$ as a smooth function of $\e,\kappa,\l,\b$. Moreover, it is clear that $\cV(\e,\kappa,\l,\beta)$ is linear in $\b$ and that $\cV(0,\kappa,\l,\b)=0$ because $(I-P)$ commutes with $B(0,\kappa,\l)$ and $\Upsilon_\b$ is annihilated by $(I-P)$. We will find it notationally convenient to define the following linear operator
	\begin{equation}\label{eq:Tlambda}
		T_\l:=\left[(I-P)B(0,\kappa,\l)(I-P)\right]^{-1} 
	\end{equation}
	It is important to observe that $T_\l$ is analytic with respect to $\l$. We now compute the Taylor expansion of $\cV$ with respect to $\e$ about $\e=0$, starting with the derivative with respect to $\e$. To do this, differentiate $(I-P)B(\e,\kappa,\l)(\Upsilon_\b+\cV)=0$ with respect to $\e$, evaluate at $\e=0$ and then solve for $\d_\e\cV(0,\kappa,\l)$.
	\begin{equation}\label{eq:dVdeps}
		\d_\e\cV(0,\kappa,\l,\b)=-T_\l (I-P)B_\e(0,\kappa,\l)\Upsilon_\b 
	\end{equation}
	Similarly, by taking two derivatives with respect to $\e$, we can solve for $\d_\e^2\cV(0,\kappa,\l,\b)$ as
	\begin{equation}\label{eq:d2Vdeps2}
		\d_\e^2\cV(0,\kappa,\l,\b)=-T_\l(I-P)B_{\e\e}(0,\kappa,\l)\Upsilon_\b-2T_\l(I-P)B_\e(0,\kappa,\l)\d_\e\cV(0,\kappa,\l)
	\end{equation}
	Next, we look at $PB(\e,\kappa,\l)W=0$. We begin this by Taylor expanding both $B$ and $\cV$ as 
	\begin{equation}\label{eq:PBTaylor}
		\begin{split}
			PB(\e,\kappa,\l)(\Upsilon_\b+\cV)=P\left[B(0,\kappa,\l)+\e B_\e(0,\kappa,\l)+\frac{1}{2}\e^2 B_{\e\e}(0,\kappa,\l)+\cO(\e^3) \right]\\ \left(\Upsilon_\b+\e\cV_\e(0,\kappa,\l,\b)+\frac{1}{2}\e^3\cV_{\e\e}(0,\kappa,\l,\b)+\cO(\e^3) \right)=0
		\end{split}
	\end{equation}
	\begin{obs}\label{obs:keycommutator}
		$PB(0,\kappa,\l)=-\l P$
	\end{obs}
	From this observation, 
	we conclude that the $\cO(1)$ term in \eqref{eq:PBTaylor} is $-\l\Upsilon_\b$ as it should be. 
	Next, we show at $\cO(\e)$ that
	\begin{equation}\label{eq:PBTaylorOE}
		PB(0,\kappa,\l)\d_\e\cV(0,\kappa,\l,\b)+PB_\e(0,\kappa,\l)\Upsilon_\b=0.
	\end{equation}
	Namely, applying observation \eqref{obs:keycommutator} again, we find that $PB(0,\kappa,\l)\d_\e\cV(0,\kappa,\l,\b)=0$ as $P\cV=0$ implies $P\d_\e\cV=0$. From \eqref{eq:DBDeps}, we find that
	\begin{equation}\label{eq:PBTaylorOE2}
		PB_\e(0,\kappa,\l)\Upsilon_\b=\kappa PL_k(k_*,0)D_\xi\Upsilon_\b+\left(k_*d_\e(0,\kappa)+\kappa d_*\right)\d_\xi\Upsilon_\b+PD^2\cN(0)(\d_\e\tilde{u}_{0,\kappa},\Upsilon_\b).
	\end{equation}
	From the existence theory, Proposition \ref{mainLS}, we know that at $\e=0$, 
	$\d_\e\tilde{u}_{0,\kappa}=\Upsilon_\a$ for a real $\a$, hence the 
	final, nonlinear term vanishes in the above equation as it is Fourier supported in $\{0,\pm2 \}$ but $P$ first projects onto Fourier modes $\pm1$. We've also used the fact that $\d_\xi$ commutes with all Fourier multiplier operators. We may now apply spectral perturbation theory to simplify \eqref{eq:PBTaylorOE2} into
	\begin{equation}\label{eq:PBTaylorOE3}
		PB_\e(0,\kappa,\l)\Upsilon_\b=\frac{1}{2}\left[\kappa\tl_{k}(k_*,0)+i(k_*d_\e(0,\kappa)+\kappa d_* ) \right]\b e^{i\xi}r+c.c.
	\end{equation}
	But the bracketed expressions is identically zero, as can be seen from the equation defining $\delta=d-d_*$ in the Lyapunov-Schmidt, reproduced below
	\begin{equation}\label{eq:dequation}
		\Im\d_k\tl(k_*,0)\kappa+\frac{1}{2}\Im\d_k^2\tl(k_*,0)\kappa^2+\Im\d_\mu\tl(k_*,0)\mu+d_*\kappa+\delta k_*+\delta\kappa+\Im n(|\a|^2;\mu,d,k)+\cO(\e^3)=0
	\end{equation}
	Now making the identification $\kappa=\e\kappa$ and $\mu\sim\e^2$ and noting that $n(|\a|^2;\mu,d,k)=\cO(\e^2)$ gives us the desired conclusion for the bracketed expression by differentiating with respect to $\e$ and evaluating at $\e=0$.

	We now come to the most important term in the expansion, $\cO(\e^2)$
	\begin{equation}\label{eq:PBTaylorOEE}
		\frac{1}{2}PB_{\e\e}(0,\kappa,\l)\Upsilon_\b+PB_\e(0,\kappa,\l)\d_\e\cV(0,\kappa,\l)+\frac{1}{2}PB(0,\kappa,\l)\d_\e^2\cV(0,\kappa,\l)
	\end{equation}
	As before observation \eqref{obs:keycommutator} implies the last term vanishes identically. We first look at the $PB_{\e\e}$ term. We may expand it by applying \eqref{eq:D2BDeps2} to find
	\begin{equation}\label{eq:PBTaylorOEEBEE}
	\begin{split}
		PB_{\e\e}(0,\kappa,\l)\Upsilon_\b=\kappa^2PL_{kk}(k_*,0)D_\xi^2\Upsilon_\b+\mu_{\e\e}PL_\mu(k_*,0)\Upsilon_\b+\left(d_{\e\e}(0,\kappa)k_*+2\kappa d_\e(0,\kappa) \right)\d_\xi\Upsilon_\b+\\+PD^2\cN(0)(\d_\e^2\tilde{u}_{0,\kappa},\Upsilon_\b)+PD^3\cN(0)(\d_\e\tilde{u}_{0,\kappa},\d_\e\tilde{u}_{0,\kappa},\Upsilon_\b)
	\end{split}
	\end{equation}
	We adopt the convention that $\cN(U)=\cQ(U,U)+\cC(U,U,U)+\cO(|U|^4)$ as in the derivation of complex Ginzburg Landau. Note that $2\cQ(U,U)=D^2\cN(0)(U,U)$ and $6\cC(U,U,U)=D^3\cN(0)(U,U,U)$ by Taylor's theorem. We begin to simplify \eqref{eq:PBTaylorOEEBEE} by first computing $\d_\e^2\tilde{u}_{0,\kappa}$. By construction, we have
	\begin{equation}\label{eq:tildeueqn}
		L(k,\mu)\tilde{u}_{\e,\kappa}+d(k,\mu)k\d_\xi\tilde{u}_{\e,\kappa}+\cQ(\tilde{u}_{\e,\kappa},\tilde{u}_{\e,\kappa})+\cO(|\tilde{u}_{\e,\kappa}|^3)=0
	\end{equation}
	Observe that the nonlinearity is $\e^2\cQ(\Upsilon_\a,\Upsilon_\a)+\cO(\e^3)$, hence to $\cO(\e^2)$ we only have Fourier modes $\{0,\pm1,\pm2 \}$. Since we're interested in $PD^2\cN(0)(\d_\e^2\tilde{u}_{0,\kappa},\Upsilon_\b)$, we only need to compute the $\{0,\pm2 \}$ modes of $\d_\e^2\tilde{u}_{0,\kappa}$. Define the following matrices for $\eta\not=\pm 1$
	\begin{equation}
		S_\eta:=\left[S(\eta k_*,0)+i\eta d_*k_* \right]^{-1}
	\end{equation}
	where $S(k,\mu)$ is the symbol of $L(k,\mu)$. Plugging in the Taylor series for $\tilde{u}_{\e,\kappa}$ and Taylor expanding the symbol in \eqref{eq:tildeueqn} shows that
	\begin{equation}
		\frac{1}{2}S(0,0)\widehat{\d_\e^2\tilde{u}_{0,\kappa}}(0)+\frac{1}{4}\a^2\left[\cQ(r,\bar{r})+\cQ(\bar{r},r) \right]=0
	\end{equation}
	or equivalently using the symmetry of $\cQ$
	\begin{equation}\label{eq:mode0}
		\widehat{\d_\e^2\tilde{u}_{0,\kappa}}(0)=-\frac{1}{2}\a^2S_0\left[\cQ(r,\bar{r})+\cQ(\bar{r},r) \right]=-\a^2S_0\cQ(r,\bar{r})
	\end{equation}
	Similarly, we have
	\begin{equation}
		\frac{1}{2}\left[S(2k_*,0)+2ik_*d_* \right]\widehat{\d_\e^2\tilde{u}_{0,\kappa}}(2)+\frac{1}{4}\a^2\cQ(r,r)=0
	\end{equation}
	equivalently
	\begin{equation}\label{eq:mode2}
		\widehat{\d_\e^2\tilde{u}_{0,\kappa}}(2)=-\frac{1}{2}\a^2S_2\cQ(r,r)
	\end{equation}
	\begin{remark}
		That \eqref{eq:mode0} and \eqref{eq:mode2} both have $\a^2$ follows from taking $\a\in\RR$, if $\a\in\CC\backslash\RR$ then they become $|\a|^2$ and $\a^2$ respectively.
	\end{remark}
	Combining \eqref{eq:upsilondef}, \eqref{eq:mode0}, \eqref{eq:mode2} allows us to compute the quadratic term in \eqref{eq:PBTaylorOEEBEE}
	\begin{equation}\label{eq:PBTaylorQuad1}
		PD^2\cN(0)(\d_\e^2\tilde{u}_{0,\kappa},\Upsilon_\b)=2P\cQ\left(-\a^2S_0\cQ(r,\bar{r})-\frac{1}{2}\a^2S_2\cQ(r,r)\e^{2i\xi}+c.c.,\frac{1}{2}\left(\b e^{i\xi}r+c.c\right) \right)
	\end{equation}
	Note that the coefficient of $e^{i\xi}$ in \eqref{eq:PBTaylorQuad1} is
	\begin{equation}\label{eq:PBTaylorQuad2}
		-\a^2\Pi\left[\cQ(S_0\cQ(r,\bar{r}),r)\b+\frac{1}{2}\cQ(S_2\cQ(r,r),\bar{r})\bar{\b} \right]
	\end{equation}
	For the cubic term in \eqref{eq:PBTaylorOEEBEE}, we have
	\begin{equation}\label{eq:PBTaylorCubic1}
		PD^3\cN(0)(\d_\e\tilde{u}_{0,\kappa},\d_\e\tilde{u}_{0,\kappa},\Upsilon_\b)=6P\cC(\Upsilon_\a,\Upsilon_\a,\Upsilon_\b)
	\end{equation}
	This has $e^{i\xi}$ coefficient
	\begin{equation}\label{eq:PBTaylorCubic2}
		\frac{6}{8}\a^2\Pi\left[\cC(r,r,\bar{r})\bar{\b}+\cC(r,\bar{r},r)\b+\cC(\bar{r},r,r)\b \right]
	\end{equation}

	Next we look at the term $PB_\e(0,\kappa,\l)\d_\e\cV(0,\kappa,\l)$
	\begin{equation}\label{eq:PBTaylorBEVE}
	\begin{split}
		&PB_\e(0,\kappa,\l)\d_\e\cV(0,\kappa,\l)=\\&
		=P\left[\kappa L_k(k_*,0)D_\xi+\left(k_*d_\e(0,\kappa)+\kappa d_*\right)\d_\xi+D^2\cN(0)(\Upsilon_\a,\cdot)\right](-T_\l (I-P)B_\e(0,\kappa,\l)\Upsilon_\b)
	\end{split}
	\end{equation}
	This consists of four terms, depending on whether the one takes the ``linear'' portion of $B_\e$ or the ``nonlinear'' portion of $B_\e$. First, we 
	compute the``linear-linear'' term 
	\ba\label{eq:linlin}
		-\kappa^2 PL_k(k_*,0)D_\xi(I-P)&T_\l(I-P)L_k(k_*,0)D_\xi\Upsilon_\b=\\
		&-\kappa^2 PL_k(k_*,0)D_\xi(I-P)N(I-P)L_k(k_*,0)D_\xi\Upsilon_\b+\cO(|\l|),
	\ea
	where we've used the fact that $T_\l:(I-P)L^2_{per}(\RR;\RR^n)\rightarrow(I-P)H^m_{per}(\RR;\RR^n)$, and so the terms containing a $\d_\xi$ are zero. 
	Here, as in \cite{WZ}, we are using the notation
	$$
	N=\left[(I-P)(S(k_*,0)+ik_*d_*)(I-P) \right]^{-1}=T_0.
	$$
	Next, we 
	compute the``nonlinear-linear'' term,
	\begin{equation}\label{eq:nonlinlin}
		PD^2\cN(0)(\Upsilon_\a,T_\l(I-P)\left[\kappa L_k(k_*,0)D_\xi+\left(k_*d_\e(0,\kappa)+\kappa d_*\right)\d_\xi\right]\Upsilon_\b )=0,
	\end{equation}
	which vanishes by the fact that it has Fourier support contained in $\{0,\pm2\}$. 
	Similarly, the linear-nonlinear term vanishes:
	\begin{equation}\label{eq:linnonlin}
		P\left[\kappa L_k(k_*,0)D_\xi+\left(k_*d_\e(0,\kappa)+\kappa d_*\right)\d_\xi\right]T_\l(I-P)D^2\cN(0)(\Upsilon_\a,\Upsilon_\b)=0.
	\end{equation}

	Finally, we have the nonlinear-nonlinear term,
	\ba\label{eq:nonlinnonlin1}
		PD^2\cN(0)(\Upsilon_\a,T_\l(I-P)D^2\cN(0)&(\Upsilon_\a,\Upsilon_\b))=\\
		& PD^2\cN(0)(\Upsilon_\a,N(I-P)D^2\cN(0)(\Upsilon_\a,\Upsilon_\b))+\cO(|\l|).
		\ea
	We start by computing 
	\ba
	D^2\cN(0)(\Upsilon_\a,\Upsilon_\b)&=\frac{1}{4}\a\big[D^2\cN(0)(r,r)\b e^{2i\xi}+D^2\cN(0)(\bar{r},r)\b\\
&\quad +D^2\cN(0)(r,\bar{r})\bar{\b} +D^2\cN(0)(\bar{r},\bar{r})\bar{\b}e^{-2i\xi} \big].
		\ea
	Thus, \eqref{eq:nonlinnonlin1} has as coefficient of $e^{i\xi}$
	\begin{equation}\label{eq:nonlinnonlin2}
		\frac{1}{8}\a^2P\left[D^2\cN(0)(\bar{r},S_2D^2\cN(0)(r,r) )\b+D^2\cN(0)(r,S_0D^2\cN(0)(r,\bar{r})(\b+\bar{\b}) ) \right],
	\end{equation}
	or, in terms of $\cQ$,
	\begin{equation}\label{eq:nonlinnonlin3}
		\frac{1}{2}\a^2P\left[\cQ(\bar{r},S_2\cQ(r,r) )\b+\cQ(r,S_0\cQ(r,\bar{r})(\b+\bar{\b})) \right].
	\end{equation}
	Combining \eqref{eq:nonlinnonlin3}, \eqref{eq:linlin}, \eqref{eq:PBTaylorCubic2}, \eqref{eq:PBTaylorQuad2}, 
	and \eqref{eq:PBTaylorOEEBEE} gives:
	\ba \label{eq:ewww}
\frac{1}{2}PB_{\e\e}&(0,\kappa,\l)\Upsilon_\b+PB_\e\d_\e\cV(0,\kappa,\l,\b)=
		\frac{1}{4}\big(\kappa^2\Pi S_{kk}(k_*,0)\b r +\mu_{\e\e}\Pi S_\mu(k_*,0)\b r\\
		&\quad +i\big(d_{\e\e}(0,\kappa)k_*+2\kappa d_\e(0,\kappa) \big)\b r \big)e^{i\xi}
			+\frac{1}{2}\big(-\a^2\Pi\big[\cQ(S_0\cQ(r,\bar{r}),r)\b\\
			&\quad +\frac{1}{2}\cQ(S_2\cQ(r,r),\bar{r})\bar{\b} \big]
			+\frac{6}{8}\a^2\Pi\big[\cC(r,r,\bar{r})\bar{\b}+\cC(r,\bar{r},r)\b+\cC(\bar{r},r,r)\b \big] \big)e^{i\xi}\\
&\quad -\frac{1}{2}\big(\kappa^2 \Pi S_k(k_*,0)(I-P)N(I-P)S_k(k_*,0)\b r+\a^2\Pi\big[\cQ(\bar{r},S_2\cQ(r,r) )\b \\
&\quad +\cQ(r,S_0\cQ(r,\bar{r})(\b+\bar{\b})) \big] \big)e^{i\xi}+\cO(|\l|)+c.c.
\ea

Note that we have the identity
	\begin{equation}\label{eq:spectralid}
		\tl_{kk}(k_*,0)r=2\Pi(\frac{1}{2}S_{kk}(k_*,0)r-S_k(k_*,0)(I-\Pi)N(I-\Pi)S_k(k_*,0)r ).
	\end{equation}
	Looking at \eqref{eq:ewww}, we can factor out $\frac{1}{2}$ and apply \eqref{eq:spectralid} to simplify the linear part to 
	\begin{equation}\label{eq:ewwwlin}
	\begin{split}
		\frac{1}{2}\left(\kappa^2\Pi S_{kk}(k_*,0)\b r+\mu_{\e\e}\Pi S_\mu(k_*,0)\b r-\kappa^2 \Pi S_k(k_*,0)(I-P)N(I-P)S_k(k_*,0)\b r \right. \\
		\left. i\left(d_{\e\e}(0,\kappa)k_*+2\kappa d_\e(0,\kappa)\right)\b r\right)e^{i\xi}+c.c.=\\
		=\frac{1}{2}\left(\tl_{kk}(k_*,0)\kappa^2+\mu_{\e\e}\tl_\mu(k_*,0)+i(d_{\e\e}(0,\kappa)+2\kappa d_\e(0,\kappa)) \right)\b re^{i\xi}+c.c.
	\end{split}
	\end{equation}

	In order to simplify the nonlinear part of \eqref{eq:ewww}, we recall that $\g$ in the complex Ginzburg-Landau can be found through the formula 
	\begin{equation}\label{eq:gammadef}
		\begin{split}
			\g=\frac{1}{8}\ell\left[3\cC(r,r,\bar{r})-4\cQ(r,S_0\cQ(r,\bar{r}))-2\cQ(S_2\cQ(r,r),\bar{r})\right],
		\end{split}
	\end{equation}
where we've used the symmetry of the forms $\cQ$ and $\cC$.
	First we collect the $\bar{\beta}$ terms in \eqref{eq:ewww} as
	\begin{equation}\label{eq:ewwwbar}
		\frac{1}{2}\bar{\b}\a^2\Pi\left(-\frac{1}{2}\cQ(S_2\cQ(r,r),\bar{r})+\frac{3}{4}\cC(r,r,\bar{r})-\cQ(r,S_0\cQ(r,\bar{r})) \right)=\g\bar{\b}\a^2r.
	\end{equation}
	Next, we look at the $\b$ part of the nonlinearity of \eqref{eq:ewww}:
	\begin{equation}\label{eq:ewwwbeta}
		\frac{1}{2}\b\a^2\Pi\left(-\cQ(S_0\cQ(r,\bar{r}),r)+\frac{12}{8}\cC(r,\bar{r},r)-\cQ(\bar{r},S_2\cQ(r,r))-\cQ(r,S_0\cQ(r,\bar{r})\right)=2\b\a^2\g r.
	\end{equation}
	Note that in both of these equations, we have ignored the ``universal'' $\frac{1}{2}$ 
multiplying all of the terms in \eqref{eq:ewww}.

	Since the $e^{i\xi}$ and $e^{-i\xi}$ modes are complex conjugates of each other, and everything is parallel to $r$ in the $e^{i\xi}$ mode, it suffices to solve for the coefficient of $e^{i\xi }$.
	Combining all of this, we find that \eqref{eq:ewww} reduces to the much nicer form
	\begin{equation}\label{eq:oeps2}
		\frac{1}{2}\left[ \frac{1}{2}\left(\tl_{kk}(k_*,0)\kappa^2+\mu_{\e\e}\tl_\mu(k_*,0)+i(d_{\e\e}(0,\kappa)k_*+2\kappa d_\e(0,\kappa)) \right)\b+2\g\a^2\b+\g\a^2\bar{\b}\right]+\cO(|\l|).
	\end{equation}

	\begin{remark}
		Although the expression in \eqref{eq:oeps2} is technically a complex scalar, it is not holomorphic in $\b$; we will find it better to solve for $\b_1:=\Re \b$ and $\b_2:=\Im \b$ by treating \eqref{eq:PBTaylor} as system of real variables.
	\end{remark}
	\begin{equation}
		\frac{1}{2}\left(\hal_{kk}(k_*,0)\kappa^2+\mu_{\e\e}\hal_\mu(k_*,0)+i(d_{\e\e}(0,\kappa)k_*+2\kappa d_\e(0,\kappa))\right)+\g \a^2=0
	\end{equation}
	by comparing to the $\cO(\e^2)$ term of the reduced equation in the Lyapunov-Schmidt reduction
	\begin{equation}\label{eq:LSreduced}
		\left(\tl_k(k_*,0)\kappa+\frac{1}{2}\hal_{kk}(k_*,0)\kappa^2+\tl_\mu(k_*,0)\mu+i(d_*\kappa+\delta k_*+\delta\kappa) \right)+n(|\a|^2;\mu,k,d)+\cO(\e^3)=0,
	\end{equation}
	writing $n(\a^2;\mu,k,d)=\g\e^2\a^2+\cO(\e^3)$, $\mu=\frac{1}{2}\mu_{\e\e}\e^2+\cO(\e^3)$, $\kappa=\e\kappa$ and then plugging in the Taylor expansion for $\delta=d-d_*$.
	We may rewrite \eqref{eq:PBTaylor} as the eigenvalue problem
	\begin{equation}\label{eq:coperiodic}
		-\l\bp \b_1 \\ \b_2 \ep +2\e^2\a^2\bp \Re\g & 0 \\ \Im\g & 0\ep \bp \b_1 \\ \b_2\ep +\cO(\e^3,\e^2|\l|)\b=0.
	\end{equation}

	\begin{remark}
		It is vitally important in the above calculations that all terms involving $d$ vanish
		identically.
		This is because in the Ginzburg-Landau expansion, 
		all information about the wave speed $d$ is ``hidden'' inside $\Hx$ and $\xi$. 
		Nowhere does the wave speed make a direct appearance in the Ginzburg-Landau or in the stability criteria.
	\end{remark}

	\begin{lemma}\label{lem:corefined}[Refined Error Estimate]
		The error in \eqref{eq:coperiodic} has the following form
		\begin{equation}\label{eq:improvederror}
			\bp \cO(\e^3,\e^2|\l|) & \cO(\e^2|\l|) \\
			\cO(\e^3, \e^2|\l|) & \cO(\e^2|\l|) \ep		
		\end{equation}
	\end{lemma}

	Before we prove the lemma, we show how it implies the following theorem.

	\begin{theorem}\label{thm:copstab}[Coperiodic Stability]
		Let $\tilde{u}_{\e,\kappa}$ be the solutions in Proposition \ref{mainLS}.
		For some $\delta>0$,
		the spectrum of $B(\e,\kappa):=L(k,\mu)+d(\e,\kappa)k\d_\xi+D\cN(\tilde{u}_{\e,\kappa})$ admits the following decomposition for all $\e$ sufficiently small and $\kappa^2\leq q_E^2$ for $q_E$ defined in \eqref{eq:CGLExistence}.
		\begin{equation}
			\s(B(\e,\kappa) )=S\cup\{\l_1,\l_2 \}
		\end{equation}
		where
		\begin{equation} \label{l1}
			\l_1(\e,\kappa)=2\e^2\a^2\Re\g+\cO(\e^3)
		\end{equation}
		and
		\begin{equation}\label{l2}
			\l_2(\e,\kappa)=0
		\end{equation}
		and, for all $\l\in S$, 
		\begin{equation}\label{ldelta}
			\Re\l<-\delta.
		\end{equation}
	\end{theorem}
	
	\begin{remark}
		Formulae \eqref{l1}--\eqref{ldelta} are consistent, as they must be, with 
		the corresponding formulae obtained in \cite[Prop. 3.2]{SZJV} for an illustrative model
		(the Brusselator equation) in the reaction diffusion case.
		Both describe transcritical bifurcations with $SO(2)$ symmetry, the latter with an additional
		reflective symmetry making it an $O(2)$ bifurcation as well.
	\end{remark}

	\begin{proof}[Proof of Theorem \eqref{thm:copstab}]
		We set the determinant of \eqref{eq:coperiodic} equal to zero, and by the refined error estimate in lemma\eqref{lem:corefined}, we find that
		\begin{equation}\label{eq:detrefinedcop}
		\begin{split}
			\det\bp -\l+2\e^2\a^2\Re\g+\cO(\e^3,\e^2|\l|) & \cO(\e^2|\l|)\\ 2\e^2\a^2\Im\g+\cO(\e^3,\e^2|\l|) & -\l+\cO(\e^2|\l|)\ep\\=(-\l+2\e^2\a^2\Re\g+\cO(\e^3,\e^2|\l|))(-\l+\cO(\e^2|\l|))-(\cO(\e^2|\l|))(2\e^2\a^2\Im\g+\cO(\e^3,\e^2|\l|))=0.
		\end{split}
		\end{equation}
		Now observe that we can factor out $-\l$ out of the above, to get
		\begin{equation}
			-\l\left((1-\cO(\e^2))(-\l+2\e^2\a^2\Re\g+\cO(\e^3,\e^2|\l|))+\cO(\e^4) \right)=0.
		\end{equation}
	\end{proof}
	\begin{proof}[Proof of Lemma \eqref{lem:corefined}]
		It suffices to show that the error at $\l=0$ takes the form
		\begin{equation}
			\bp \cO(\e^3) & 0\\
				\cO(\e^3) & 0 \ep.
		\end{equation}
		In the existence analysis of \cite{WZ}, in the course of showing \eqref{soln}, there was
		established in fact the more detailed expansion
		\be\label{reform}
		\hbox{\rm $\tilde{u}_{\e,\kappa}=\e\Upsilon_\a+\e V(\e,\kappa,\a)$,  where $PV=0$.} 
		\ee
		of which we shall make use now, where $\Upsilon_\a$ is as defined in \eqref{eq:upsilondef}.
		Note that formulation \eqref{reform} represents a slight shift in notation relative to \cite{WZ}.
		There, we wrote $V=\e\Upsilon_\a$ and $U=V+W+X$, where $\supp(\hat{W} )\subset\{\pm1 \}$ and $\hat{X}(1)=0$. Here, we've relabeled and our new $V$ is $W+X$ and the old $V$ is now $\Upsilon_a$. 
		The reason for this notational shift is to relate the quantities in the existence analysis
		with the corresponding quantities in our stability analysis, through the following claim.
		\begin{claim}
			\begin{equation}
				\d_{\b_j}\cV(\e,\kappa,0,\beta)=\d_{\a_j}V(\e,\kappa,\a)|_{\a_2=0} 
			\end{equation}
			for $j=1,2$.
		\end{claim}
		Before proving the claim, note that since $\cV$ is linear in $\b$ and so the right-hand side of the equality in the claim is constant in $\b$. Moreover, since $\a$ is determined by the parameters $\e,\kappa$; it follows that the two sides actually only depend on $(\e,\kappa)$. We begin by comparing the equations
		that $V$ and $\cV$ satisfy, respectively:
		\begin{equation}\label{eq:system}
		\begin{aligned}
			&(I-P)\left(L(k,\mu)+d(k,\mu)k\d_\xi+D\cN(\tilde{u}_{\e,\kappa}) \right) (\Upsilon_\b+\cV)=0,\\
			&(I-P)\left(L(k,\mu)(\e \Upsilon_\a+\e V)+d(k,\mu)k\d_\xi(\e\Upsilon_\a+\e V)+\cN((\e\Upsilon_\a+\e V)) \right)=0.
		\end{aligned}
		\end{equation}

		Differentiating the first equation of \eqref{eq:system} with respect to $\b_1$ and the second with respect to $\a_1$ and then evaluating at $\a_2=0$ gives
		\begin{equation}\label{eq:dsystem}
		\begin{aligned}
			&(I-P)\left(L(k,\mu)+d(k,\mu)k\d_\xi+D\cN(\tilde{u}_{\e,\kappa}) \right) (\Upsilon_1+\d_{\b_1}\cV)=0,\\
			&(I-P)\left(L(k,\mu)+d(k,\mu)k\d_\xi+D\cN(\tilde{u}_{\e,\kappa}) \right)(\Upsilon_1+\d_{\a_1}V|_{\a_2=0})=0.
		\end{aligned}
		\end{equation}
		By the uniqueness statement of the inverse function theorem and a similar argument for the $\b_2$ and $\a_2$ derivatives, the claim follows. 
		
		We now look at the other equation from the existence problem:
		\begin{equation}\label{eq:existencePeqn}
			P(L(k,\mu)+d(k,\mu)k\d_\xi)(\e \Upsilon_\a+\e V)+P\cN(\e \Upsilon_\a+\e V)=0,
		\end{equation} 
		and differentiate with respect to $\a_1$ and evaluate at $\a_2=0$ to get
		\begin{equation}
			PB(\e,\kappa,0)(\Upsilon_1+\d_{\b_1}\cV(\e,\kappa,0,\b))=0
		\end{equation}
		by the claim and a similar equation for the $\a_2$ derivative. 
	
This means that we can get the coefficients of the reduced spectral equation at $\l=0$ by 
differentiating the reduced equation from the existence problem with respect to $\a_1$ and $\a_2$, 
		then setting $\a_2=0$ as the amplitude was taken to be real. 
		We recall that \eqref{eq:existencePeqn} reduced to \eqref{eq:LSreduced}: 
		\begin{equation*}
			\left(\tl_k(k_*,0)\kappa+\frac{1}{2}\tl_{kk}(k_*,0)\kappa^2+\tl_\mu(k_*,0)\mu+i(d_*\kappa+\delta k_*+\delta\kappa) \right)\a+n(|\a|^2;\mu,k,d)\a+\cO(\e^3\a)=0,
		\end{equation*}
	and that $n(|\a|^2;\mu,k,d)=\e^2\g|\a|^2+\cO(|\a|^4)$. 
		
	Denoting $z=x+iy\in\CC$ we let $[[z]]\in M_2(\RR)$ be the matrix
		\begin{equation}
			[[z]]=\bp x & -y \\ y & x\ep,
		\end{equation}
		so that \eqref{eq:LSreduced} takes the form
		\begin{equation}\label{eq:LSreduced2}
			[[\tl_k(k_*,0)\kappa+\frac{1}{2}\tl_{kk}(k_*,0)\kappa^2+\tl_\mu(k_*,0)\mu+i(d_*\kappa+\delta k_*+\delta\kappa)+\e^2\g(\a_1^2+\a_2^2)+\cO(\e^3) ]]\bp \a_1 \\ \a_2\ep=0.
		\end{equation}

		Let $\cA(\e,\kappa)$ be the complex scalar
		\begin{equation}\label{eq:mathcalA}
			\cA(\e,\kappa)=\tl_k(k_*,0)\kappa+\frac{1}{2}\tl_{kk}(k_*,0)\kappa^2+\tl_\mu(k_*,0)\mu+i(d_*\kappa+\delta k_*+\delta\kappa).
		\end{equation}
		We differentiate \eqref{eq:LSreduced2} with respect to $\a_1$ and then evaluate at $\a_2=0$ to get
		\begin{equation}
			\bp \Re\cA+3\e^2\Re\g\a_1^2+\cO(\e^3)\\
			\Im\cA+3\e^2\Im\g\a_1^2+\cO(\e^3) \ep.
		\end{equation}

		Note that the $\a_2$ derivative of \eqref{eq:LSreduced2} vanishes at $\a_2=0$ because if we call the matrix $M(\a;\e,\kappa)$ then we note $\d_{\a_2}M(\a;\e,\kappa)=\cO(\a_2)$ by applying the chain rule to $n(|\a|^2;\e,\kappa)$ and noting that $\d_{\a_2}|\a|^2=\cO(\a_2)$. On the other hand, $M(\a_1;\e,\kappa)=0$ by choice of $\a_1$ and the form of $M$. To finish the argument, we use the observation that $\cA+\e^2\g\a_1^2=\cO(\e^3)$.
	\end{proof}

	\begin{remarks}
		1. Notice that in the proof, we never used the explicit form of $L(k,\mu)$; only the symbol appeared. This suggests that in the case of pseudodifferential operators, one would still have spectral stability; at least in the coperiodic case.

		2. In $O(2)$ invariant systems, we get the stronger error estimate
		$
			\bp \cO(\e^3,\e^2|\l|) & \cO(\e^2|\l|)\\
			\cO(\e^2|\l|) & \cO(\e^2|\l|) \ep
			$
since the reduced equation in Lyapunov-Schmidt is a real equation. Note, however, that even the $O(2)$ invariant estimate is still weaker than the one found by explicit computation for the Brusselator model in \cite[Prop. 3.1]{SZJV}.
	\end{remarks}


	\section{General stability}\label{sec:gencase}
	We now turn to the rigorous stability analysis with respect to general perturbations, that is,
determination of	
	the spectrum of $\cL^{\eps, \kappa}$ considered as an operator on the whole line.
	Accordingly, we define a full Bloch-type operator depending additionally on Floquet number $\s$ as
	\begin{equation}\label{eq:BlochDef}
		B(\e,\kappa,\l,\s):=L(k,\mu;\s)+dk\d_{\xi}+i\s\left(kd_*+\frac{kk_*d_\e(0,\kappa)}{\kappa}\right)
		+D\cN(\tilde{u}_{\e,\kappa})-\l,
	\end{equation}
	where
	\begin{equation}
		L(k,\mu;\s)U(\xi)=\sum_{\eta\in\ZZ}S(k(\eta+\s),\mu)\hat{U}(\eta)e^{i\eta\xi}.
	\end{equation}

		This is not the usual Bloch operator, which is given by
		\begin{equation*}
			\tilde{B}(\e,\kappa,\l,\s)=L(k,\mu;\s)+dk\d_\xi+i\s dk+D\cN(\tilde{u}_{\e,\kappa})-\l.
		\end{equation*}
		Notice that the only difference between the two operators is that $B$ and $\tilde{B}$ have different constants multiplying $i\s$, which evidently does not change stability properties, but only shifts
		spectral curves $\lambda(\sigma)$ in imaginary direction.  
		This change is made so that the spectral curves obtained from the ultimate reduced 
		equations obtained by Lyapunov-Schmidt reduction will match those obtained
		through the treatment of linearized complex Ginzburg-Landau stability in Section
		\ref {sec:cglstability}.
See remark \eqref{rem:BlochDiff} for further details about how this change in constants affects the reduced equation.

	As in the coperiodic case, we will define the following operators that arise from taking $\e$ and $\s$ derivatives of $B$.
	\begin{align}\label{eq:derivop2}
		L_k(k,\mu;\s)U(\xi)&=\sum_{\eta\in\ZZ} S_k(k(\eta+\s),\mu)\hat{U}(\eta)e^{i\eta\xi}\\
		L_{kk}(k,\mu;\s)U(\xi)&=\sum_{\eta\in\ZZ} S_{kk}(k(\eta+\s),\mu)\hat{U}(\eta)e^{i\eta\xi}\\
		L_\mu(k,\mu;\s)U(\xi)&=\sum_{\eta\in\ZZ} S_\mu(k(\eta+\s),\mu)\hat{U}(\eta)e^{i\eta\xi}\\
	\end{align}
	and observe that $L_\s(k,\mu;\s)=kL_k(k,\mu;\s)$ via the chain rule.

	Note that this new Bloch operator agrees with the operator found in the analysis of the
	Brusselator model in \cite{SZJV}, or more generally in any nonconvective Turing bifurcation where $d_\e(\e,\kappa)\equiv0$. But this is in general an affine shift of the usual Bloch parameters which preserves the real parts. In convective Turing bifurcations, the usual Bloch variables produce spectral curves that do not always have the same imaginary part compared to those found in the complex Ginzburg-Landau.

	Before doing the computations, we explain the form of the Bloch operator provided in \eqref{eq:BlochDef}. To start, let $\xi_0:=kx$, $\xi_p:=k(x-dt)$ and 
	\begin{equation*}
	\xi_g:=\frac{k}{\kappa}\left(\frac{\d\xi_p}{\d\e}\right)|_{\e=0}=\frac{k}{\kappa}(\kappa x-(\kappa d(0,\kappa)+k_*d_\e(0,\kappa))t)=\frac{k}{\e}\Hx.
	\end{equation*}
	Note that the second derivative of $\xi_p$ with respect to $\e$ is proportional to $\Omega t$ and hence is not proportional to $\kappa$ in general. Hence it is important in the definition of $\xi_g$ that $\frac{\d\xi_p}{\d\e}$ is evaluated at $\e=0$ because otherwise, we are shifting the spectrum by arbitrarily large pure imaginary shifts as $\kappa\rightarrow 0$.

	Writing $\tilde{u}_{\e,\kappa}(\xi_0,t)$, we observe that it solves
	\begin{equation}
		\frac{\d\tilde{u}_{\e,\kappa}}{\d t}=L(k,\mu)\tilde{u}_{\e,\kappa}+\cN(\tilde{u}_{\e,\kappa}),
	\end{equation}
	where $L(k,\mu)$ only acts on the $\xi_0$ variable. Linearizing about $\tilde{u}_{\e,\kappa}$, we find that perturbations solve the equation
	\begin{equation}\label{eq:lineqnlab}
		\frac{\d v}{\d t}=L(k,\mu)v+D\cN(\tilde{u}_{\e,\kappa})v.
	\end{equation}
	We then write our perturbation as
	\begin{equation}\label{eq:pert}
		v(\xi_0,t)=e^{i\s\xi_g+\l t}W(\xi_p),
	\end{equation}
	where $W$ is $2\pi$-periodic in $\xi_p$. 
	Plugging \eqref{eq:pert} into \eqref{eq:lineqnlab}, we find that $W$ is an eigenfunction of the Bloch operator defined in \eqref{eq:BlochDef}. From now on, we will drop the subscript on $\xi_p$ because $\xi_0$ and $\xi_g$ will not be needed again. 
	Since we are working on one period $\xi\in [0,2\pi)$,
	we may take without loss of generality $|\s|<\frac{1}{2}$. 

	We begin solving $B(\e,\kappa,\l,\s)W=0$ by splitting $W=\Upsilon_\b+\cV$ for some $\b\in\CC$ where $P\cV=0$.
	\begin{proposition}\label{prop:CVExist}
		Consider the equation
		\begin{equation}\label{eq:IPEqn}
			(I-P)B(\e,\kappa,\l,\s)W=0.
		\end{equation}
		Then there is a unique smooth function $\cV=\cV(\e,\kappa,\l,\s,\b)$, defined for small $\e,\l$, every $|\s|<\frac{1}{2}$, and all $\b\in\CC$. Moreover, $\cV$ is linear in $\b$ and satisfies 
		\begin{enumerate}
			\item $\cV(0,\kappa,\l,0,\b)\equiv0$
			\item $\cV_\e(0,\kappa,\l,0,\b)=-T_\l (I-P)B_\e(0,\kappa,\l,0)\Upsilon_\b$
			\item $\cV_\s(0,\kappa,\l,0,\b)=-T_\l (I-P)B_\s(0,\kappa,\l,0)\Upsilon_\b$
		\end{enumerate}
		where $T_\l$ is the linear operator defined in \eqref{eq:Tlambda}.
	\end{proposition}
	\begin{proof}
		Note that \eqref{eq:IPEqn} is equivalent to
		\begin{equation}\label{eq:IPEqn2}
			(I-P)B(\e,\kappa,\l,\s)(I-P)\cV=-(I-P)B(\e,\kappa,\l,\s)\Upsilon_\b.
		\end{equation}
		To get the existence of $\cV$, it suffices to show that $(I-P)B(\e,\kappa,\l,\s)(I-P)$ is invertible for $\e,\l$ small and $|\s|<\frac{1}{2}$. We have that $B(0,\kappa,0,0)=L(k_*,0;0)+id_*k_*\d_\xi$ and $P$ is the projection onto the kernel of this operator, hence we get invertibility for $\e,\l,\s$ small.
		The desired properties follow from similar calculations to the one done in the coperiodic case, and so the details will be omitted.
	\end{proof}

	We reduce to small $\s$ when $\e$ is small, which will allow us to use Taylor expansion arguments safely.

	\begin{proposition}\label{prop:smallsigma}
		To show stability or instability for all $|\s|\leq\frac{1}{2}$, it suffices to show the corresponding property for $|\s|\ll 1$.
	\end{proposition}
	\begin{proof}
		Note that $L(k,\mu)$ is the linearization about $u=0$, and by the Turing hypotheses it has stable spectrum outside a small open set centered around $\{\pm (k_*,0) \}$. Hence if $|\s|$ is large enough, it follows that $L(k,\mu;\s)$; or equivalently $L(k,\mu;\s)+d(\e,\kappa)\d_\xi$, has stable spectrum.

		By the ellipticity of $L(k,\mu)$, we have for $\l\in\rho(L(k,\mu;\s)$ that $(\l-L(k_,\mu;\s))^{-1}:L^2_{per}([0,2\pi];\RR^n)\rightarrow H^s_{per}([0,2\pi];\RR^n)$ is a bounded operator. In addition, for all $\e,\kappa$ such that $\tilde{u}_{\e,\kappa}$ exists we have that $D\cN(u_{\e,\kappa}):H^s_{per}([0,2\pi];\RR^n)\rightarrow L^2_{per}([0,2\pi];\RR^n)$ is a bounded operator. We denote $L(k,\mu;\s)+d(\e,\kappa)\d_\xi+D\cN(\tilde{u}_{\e,\kappa})$ by $\cL(\e)$ and the corresponding resolvent by $R(\e,\l$), and rewrite $\l-\cL(\e)$ as $\l-\cL(0)-\Delta\cL(\e)=(\l-\cL(0))(Id-R(0,\l)\Delta\cL(\e))$. Since $\Delta\cL(\e)$ is bounded from $H^s_{per}([0,2\pi];\RR^n)\rightarrow L^2_{per}([0,2\pi];\RR^n)$ with norm $\cO(\e)$ and $R(0,\l):L^2_{per}([0,2\pi];\RR^n)\rightarrow H^s_{per}([0,2\pi];\RR^n)$ is bounded, we find for $\e$ sufficiently small that $(Id-R(0,\l)\Delta\cL(\e)):H^s_{per}([0,2\pi];\RR^n)\rightarrow H^s_{per}([0,2\pi];\RR^n)$ is invertible by expanding in a Neumann series. In particular, $R(\e,\l)=R(0,\l)+\cO(\e)$. A similar calculation gives continuity of the resolvent about other $\e_0$. From this we conclude spectral continuity of $B(\e,\kappa,\l,\s)$. 
		By the spectral continuity argument above, we see that $L(k,\mu;\s)+d(\e,\kappa)\d_\xi+D\cN(\tilde{u}_{\e,\kappa})$ has stable spectrum for $|\s|>=\s_0>0$ for $\e$ sufficiently small.
	\end{proof}
	\begin{remark}
		When $L(k,\mu)$ is a differential operator and $\cN$ is a local nonlinearity, one can replace the direct verification of spectral continuity above with a simpler Evans function argument \cite{G}.
	\end{remark}
	Next, we look at the equation
	\begin{equation}\label{eq:PEqn}
		PB(\e,\kappa,\l,\s)(\Upsilon_\b+\cV)=0.
	\end{equation}
	Much as in the coperiodic case, we Taylor expand this equation to second order in $\e$. Motivated by the form of \eqref{eq:hopefulreducedeqn}, 
	we Taylor expand \eqref{eq:PEqn} to second order in $\s$ as well. 
	Symbolically, we have then the expansion
	\begin{equation}\label{eq:GenPBTaylor}
	\begin{split}
		P\left[B(0,\kappa,\l,0)+\e B_\e(0,\kappa,\l,0)+\s B_\s(0,\kappa,\l,0)+\right. \\ \left. \frac{1}{2}\left(\e^2 B_{\e\e}(0,\kappa,\l,0)+2\e\s B_{\e\s}(0,\kappa,\l,0)+\s^2 B_{\s\s}(0,\kappa,\l,0) \right)+\cO(\e^3,\e^2\s,\e\s^2,\s^3) \right]\\ \left[\Upsilon_\b+\e \cV_\e(0,\kappa,\l,0,\b)+\s \cV_\s(0,\kappa,\l,0,\b)+\right. \\ \left. \frac{1}{2}\left(\e^2 \cV_{\e\e}(0,\kappa,\l,0,\b)+2\e\s \cV_{\e\s}(0,\kappa,\l,0,\b)+\s^2 \cV_{\s\s}(0,\kappa,\l,0,\b) \right)+\cO(\e^3,\e^2\s,\e\s^2,\s^3) \right]=0.
	\end{split}
	\end{equation}
	From \eqref{eq:GenPBTaylor}, we collect powers of $\e$ and $\s$, and simplify terms in the following sequence of lemmas.
	\begin{lemma}\label{lem:GTaylorO1}
		$PB(0,\kappa,\l,0)\Upsilon_\b$ in reduced form is given by
			$
			-\l \Upsilon_\b.
			$
	\end{lemma}
	\begin{proof}
		This is follows from an application of the observation in \eqref{obs:keycommutator}.
	\end{proof}

	\begin{lemma}\label{lem:GTaylorOE}
		$PB_\e(0,\kappa,\l,0)\Upsilon_\b+PB(0,\kappa,\l,0)\cV_\e(0,\kappa,\l,0,\b)$ vanishes identically.
	\end{lemma}
	\begin{proof}
		We have by observation \eqref{obs:keycommutator}
		\begin{equation}\label{eq:GTOE1}
			PB_\e(0,\kappa,\l,0)\Upsilon_\b+PB(0,\kappa,\l,0)\cV_\e(0,\kappa,\l,0,\b)=PB_\e(0,\kappa,\l,0)\Upsilon_\b.
		\end{equation}
		Expanding the operator in \eqref{eq:GTOE1}, we find that
		\begin{equation}\label{eq:GTOE2}
			P\left[\kappa L_k(k_*,0;0)D_\xi+\left(\kappa d_*+k_* d_\e(0,\kappa) \right)\d_\xi+D^2\cN(0)(\Upsilon_\a,\cdot) \right]\Upsilon_\b.
		\end{equation}
		But $L_k(k_*,0;0)$ is the same operator as $L_k(k_*,0)$ from the coperiodic case, so as before we know that this vanishes.
	\end{proof}

	\begin{lemma}\label{lem:GTaylorOS}
		$PB_\s(0,\kappa,\l,0)\Upsilon_\b+PB(0,\kappa,\l,0)\cV_\s(0,\kappa,\l,0,\b)$ also vanishes identically.
	\end{lemma}

	\begin{proof}
		As before, by the observation in \eqref{obs:keycommutator} we have
		\begin{equation}\label{eq:GTOS1}
			PB_\s(0,\kappa,\l,0)\Upsilon_\b+PB(0,\kappa,\l,0)\cV_\s(0,\kappa,\l,0,\b)=PB_\s(0,\kappa,\l,0)\Upsilon_\b.
		\end{equation}
		Writing out $B_\s$, we find that
		\begin{equation}\label{eq:GTOS2}
			PB_\s(0,\kappa,\l,0)\Upsilon_\b=P\left[k_*L_k(k_*,0;0)+i\left(d_*k_*+\frac{k_*^2d_\e(0,\kappa)}{\kappa}\right) \right]\Upsilon_\b.
		\end{equation}
		By definition of $L_k(k_*,0;0)$ and $P$ we may expand
		\ba\label{eq:GTOS3}
			PB_\s(0,\kappa,\l,0)\Upsilon_\b=
\frac{1}{2}\big(\b e^{i\xi}\big[k_* \ell S_k(k_*,0)r 
&+i\big(d_*k_*+\frac{k_*^2d_\e(0,\kappa)}{\kappa}\big) \big]r\\
		&\quad 
+\bar{\b} e^{-i\xi}\big[k_* \bar{\ell} S_k(-k_*,0)\bar{r}
		 +i\big(d_*k_*+\frac{k_*^2d_\e(0,\kappa)}{\kappa}\big) \big]\bar{r}\big).
			\ea
		Note that $S_k(-k_*,0)=-\overline{S_k(k_*,0)}$, and $\ell S_k(k_*,0)r=\hal_k(k_*,0)$ so this reduces to
		\begin{equation}\label{eq:GTOS4}
			PB_\s(0,\kappa,\l,0)\Upsilon_\b=
			k_*\big(\tl_k(k_*,0)+i\big(d_*+\frac{k_*d_\e(0,\kappa)}{\kappa}\big) \big)\Upsilon_\b=0.
		\end{equation}
	\end{proof}

	\begin{lemma}\label{lem:GTaylorOEE}
		In reduced form, $PB_{\e\e}(0,\kappa,\l,0)\Upsilon_\b+2PB_\e(0,\kappa,\l,0)\cV_\e(0,\kappa,\l,0,\b)$ is given by the following expression
			$$
		 2\a^2 \bp \Re\g & 0 \\ \Im\g & 0 \ep \bp \b_1\\ \b_2\ep +\bp \cO(\e,|\l|) & \cO(|\l|) \\ \cO(\e,|\l|) & \cO(|\l|)\ep \bp \b_1\\ \b_2\ep.
		 $$
	\end{lemma}

	\begin{proof}
		It suffices to note that $PB_{\e\e}(0,\kappa,\l,0)\Upsilon_\b+2PB_\e(0,\kappa,\l,0)\cV_\e(0,\kappa,\l,0,\b)$ agrees with the corresponding term in the coperiodic case.
	\end{proof}

	\begin{lemma}\label{lem:GTaylorOES}
		In reduced form, $PB_{\e\s}(0,\kappa,\l,0)\Upsilon_\b+PB_\e(0,\kappa,\l,0)\cV_\s(0,\kappa,\l,0,\b)+PB_\s(0,\kappa,\l,0)\cV_\e(0,\kappa,\l,0,\b)$ is given by
		\begin{equation}
			-i[[i\kappa k_*\tl_{kk}(k_*,0)]]\bp \b_1 \\ \b_2\ep+\cO(|\l|).
		\end{equation}
	\end{lemma}

	\begin{proof}
		We note that the expression in the lemma is the coefficient of $\e\s$ in \eqref{eq:GenPBTaylor}.
		Strictly speaking, there should be also a $PB(0,\kappa,\l,0)\cV_{\s\e}(0,\kappa,\l,0,\b)$ term,
		but this vanishes for the same reason as in the previous three lemmas. 
		First, we compute $B_{\e\s}(0,\kappa,\l,0)$ as
		\begin{equation}\label{eq:BES}
			B_{\e\s}(0,\kappa,\l,0)=\kappa L_k(k_*,0;0)+k_*\kappa L_{kk}(k_*,0;0)D_\xi +i(\kappa d_*+k_* d_\e(0,\kappa)).
		\end{equation}
		Consider $PB_{\e\s}(0,\kappa,\l,0)\Upsilon_\b$. From \eqref{eq:BES}, we find
		\ba\label{eq:PBES1}
			PB_{\e\s}(0,\kappa,\l,0)\Upsilon_\b&=i(\kappa d_*+k_* d_\e(0,\kappa))\Upsilon_\b
				\\ &\quad 
			+ P(\kappa L_k(k_*,0;0)
				+k_*\kappa L_{kk}(k_*,0;0)D_\xi)\Upsilon_\b.
				\ea
		For the second line, we use the definitions in \eqref{eq:derivop2} to compute
		\ba
			P(\kappa L_k(k_*,0;0)+k_*\kappa L_{kk}(k_*,0;0)D_\xi)\Upsilon_\b&=
			\frac{1}{2}\kappa \b e^{i\xi}r\big[\ell S_k(k_*,0)r+k_*\ell S_{kk}(k_*,0)r \big]\\
			& \quad +
			\frac{1}{2}\kappa \bar{\b} e^{-i\xi}\bar{r}\big[\bar{\ell} S_k(-k_*,0)\bar{r}-k_*\bar{\ell} S_{kk}(-k_*,0)\bar{r} \big]\\
			&=
			\kappa\tl_k(k_*,0)\Upsilon_\b+\frac{1}{2}\kappa k_*\ell S_{kk}(k_*,0)r \b e^{i\xi }r-c.c.
			\ea
		So, our final form of \eqref{eq:PBES1} is
		\begin{equation}\label{eq:PBES2}
			PB_{\e\s}(0,\kappa,\l,0)\Upsilon_\b=\frac{1}{2}\kappa k_*\ell S_{kk}(k_*,0)r \b e^{i\xi }r-c.c.
		\end{equation}

		Next, we look at the terms involving the first derivatives of $\cV$. First, we 
		compute
		\begin{equation}\label{eq:PBESES1}
			PB_\e(0,\kappa,\l,0)\cV_\s(0,\kappa,\l,0,\b)=-PB_\e(0,\kappa,\l,0)(I-P)T_\l(I-P)B_\s(0,\kappa,\l,0)\Upsilon_\b,
		\end{equation}
		where we've used the third identity in proposition \eqref{prop:CVExist}. Expanding $B_\e(0,\kappa,\l,0)$ and $B_\s(0,\kappa,\l,0)$, we find
		\ba\label{eq:PBESES2}
			PB_\e(0,\kappa,\l,0)\cV_\s(0,\kappa,\l,0,\b)&=-P\big[\kappa L_k(k_*,0;0)D_\xi+\big(\kappa d_*+k_* d_\e(0,\kappa) \big)\d_\xi\\
			&\quad +D^2\cN(0)(\Upsilon_\a,\cdot) \big] (I-P)T_\l(I-P)\big[k_*L_k(k_*,0;0)\big]\Upsilon_\b.
		\ea
		Immediately, we see that $(I-P)i\left(d_*k_*+\frac{k_*^2d_\e(0,\kappa)}{\kappa}\right)\Upsilon_\b=0$, since $(I-P)\Upsilon_\b=0$. Similarly, $P\left(\kappa d_*+k_* d_\e(0,\kappa) \right)\d_\xi(I-P)=0$, 
		since $\d_\xi$ commutes with all Fourier multiplier operators. This reduces \eqref{eq:PBESES2} to
			\ba\label{eq:PBESES3}
			PB_\e(0,\kappa,\l,0)\cV_\s(0,\kappa,\l,0,\b)&=-P\big[\kappa L_k(k_*,0;0)D_\xi\\
			&\quad +D^2\cN(0)(\Upsilon_\a,\cdot) \big] (I-P)T_\l(I-P)k_*L_k(k_*,0;0)\Upsilon_\b.
			\ea

		Next, we observe that since $(I-P)T_\l(I-P)k_*L_k(k_*,0;0)$ is a Fourier multiplier operator, the Fourier support of 
		$
			D^2\cN(0)(\Upsilon_\a,(I-P)T_\l(I-P)k_*L_k(k_*,0;0)\Upsilon_\b)
			$
		is contained in $\{0,\pm2 \}$, 
		which is in the kernel of $P$. Combining this with Taylor's theorem gives
		\begin{equation}\label{eq:PBESES4}
			PB_\e(0,\kappa,\l,0)\cV_\s(0,\kappa,\l,0,\b)=-\frac{1}{2}\b\kappa k_*\left[\ell S_k(k_*,0)NS_k(k_*,0)r \right]e^{i\xi }r-c.c.+\cO(|\l|),
		\end{equation}
		where $N=\left[(I-\Pi)S(k_*,0)+id_*k_*(I-P) \right]^{-1}$. An essentially identical computation reveals that
		\begin{equation}\label{eq:PBESSE}
				PB_\s(0,\kappa,\l,0)\cV_\e(0,\kappa,\l,0,\b)=-\frac{1}{2}\b\kappa k_*\left[\ell S_k(k_*,0)NS_k(k_*,0)r \right]e^{i\xi }r-c.c.+\cO(|\l|).
		\end{equation}

		Summing \eqref{eq:PBES2}, \eqref{eq:PBESES4} and \eqref{eq:PBESSE}, we get
		\begin{equation}
			\frac{1}{2}\kappa k_*\b\left[\ell S_{kk}(k_*,0)r-2\ell S_k(k_*,0)NS_k(k_*,0)r \right]e^{i\xi}r-c.c.
		\end{equation}
		From the spectral identity \eqref{eq:spectralid}, we have that this is equal to
		\begin{equation}\label{eq:OESRed1}
			\frac{1}{2}2\kappa k_*\tl_{kk}(k_*,0)\b e^{i\xi}r-c.c.
		\end{equation}
		To make this look like the corresponding term in the complex Ginzburg Landau computation, we note that
		$
			\frac{1}{2}(z-\bar{z})=-i\frac{1}{2}(iz+\bar{iz}),
			$
		which gives in \eqref{eq:OESRed1} 
		\begin{equation}\label{eq:OESRed2}
			\frac{1}{2}2\kappa k_*\tl_{kk}(k_*,0)\b e^{i\xi}r-c.c.=-i\left(\frac{1}{2}2i\kappa k_*\tl_{kk}(k_*,0)\b e^{i\xi}r+c.c.\right).
		\end{equation}

		This gives the claimed equality.
	\end{proof}

	\begin{lemma}\label{lem:GTaylorOSS}
		In reduced form, $\frac{1}{2}PB_{\s\s}(0,\kappa,\l,0)\Upsilon_\b+PB_\s(0,\kappa,\l,0)\cV_\s(0,\kappa,\l,0,\b)$ is given by
		\begin{equation}
			[[\frac{1}{2}k_*^2\tl_{kk}(k_*,0)]]\bp \b_1 \\ \b_2 \ep+\cO(|\l|)\bp \b_1 \\ \b_2\ep.
		\end{equation}
	\end{lemma}

	\begin{proof}
		We note that this expression is the coefficient of $\s^2$ in \eqref{eq:GenPBTaylor}, because as always the term involving the highest order derivative of $\cV$ vanishes identically.

		We start with the $\Upsilon_\b$ term:
		\begin{equation}
			\frac{1}{2}PB_{\s\s}(0,\kappa,\l,0)\Upsilon_\b=\frac{1}{2}P\left[k_*^2L_{kk}(k_*,0;0)\right]\Upsilon_\b.
		\end{equation}
		Using the definition of the operator $L_{kk}$ in terms of its symbol, we get
		\begin{equation}\label{eq:OSSUp}
			\frac{1}{2}P\left[k_*^2L_{kk}(k_*,0;0)\right]\Upsilon_\b=\frac{1}{2}\b \left[\frac{1}{2}k_*^2\ell S_{kk}(k_*,0)r \right]e^{i\xi}r+c.c.
		\end{equation}

		For the other term, we have
		\begin{equation}\label{eq:OSSCV1}
			PB_\s(0,\kappa,\l,0)\cV_\s(0,\kappa,\l,0,\b)=-PB_\s(0,\kappa,\l,0)(I-P)T_\l(I-P) B_\s(0,\kappa,\l,0)\Upsilon_\b.
		\end{equation}
		Upon expanding the Bloch operators, we get
		\ba\label{eq:OSSCV2}
			P\big[k_*L_k(k_*,0;0)
			&+i\big(d_*k_*+\frac{k_*^2d_\e(0,\kappa)}{\kappa}\big) \big](I-P)T_\l(I-P)\big[k_*L_k(k_*,0;0)
			\\&
			+i\big(d_*k_*+\frac{k_*^2d_\e(0,\kappa)}{\kappa}\big) \big]\Upsilon_\b.
			\ea
		Similaly as in the calculations in the proof of Lemma \eqref{lem:GTaylorOES}, 
		\eqref{eq:OSSCV2} collapses to
		\begin{equation}\label{eq:OSSCV3}
			P\left[k_*L_k(k_*,0;0)\right](I-P)T_\l(I-P)\left[k_*L_k(k_*,0;0)\right]\Upsilon_\b.
		\end{equation}
		Applying the definition of the operators $L_k$ and $T_\l$ and Taylor expanding we get
		\begin{equation}\label{eq:OSSCV4}
			\frac{1}{2}\b\left[\ell S_k(k_*,0)NS_k(k_*,0)r \right]e^{i\xi}r+c.c.+\cO(|\l|).
		\end{equation}
		Combining \eqref{eq:OSSUp} and \eqref{eq:OSSCV4} and applying the spectral 
		identity \eqref{eq:spectralid}, we get the equality claimed in the lemma.
	\end{proof}

	Overall, we've shown the following theorem.

	\begin{theorem}\label{thm:ReducedEqn}
		In reduced form, 
			$PB(\e,\kappa,\l,\s)(\Upsilon_\b+\cV(\e,\kappa,\l,\s,\b))=0$
		is given by
		\begin{equation}\label{eq:GStableRedEqn}
			\left[-\l+2\e^2\a^2 \bp \Re\g & 0 \\ \Im\g & 0 \ep+\s^2[[\frac{1}{2}k_*^2\tl_{kk}(k_*,0)]]-\e\s i[[i\kappa k_*\tl_{kk}(k_*,0)]]+h.o.t.\right]\bp \b_1 \\ \b_2\ep=0.
		\end{equation}
	\end{theorem}
	\begin{remark}\label{rem:BlochDiff}
		If one instead looks at the usual Bloch operator
		\begin{equation*}
			\tilde{B}(\e,\kappa,\s,\l)=L(k,\mu;\s)+dk\d_\xi+i\s dk+D\cN(\tilde{u}_{\e,\kappa})-\l
		\end{equation*}
		one has to add a term $-i\s(1+\e)\frac{k_*^2d_\e(0,\kappa)}{\kappa}$ to \eqref{eq:GStableRedEqn}, which is morally the same term that needed to be added to $\Hx$ in the derivation of complex Ginzburg-Landau \cite{WZ}.
	\end{remark}

	In order to solve \eqref{eq:GStableRedEqn} for $\l=\l(\e,\kappa,\s)$, we rewrite it as
	\begin{equation}\label{eq:GStableRedEqnV2}
	\begin{split}
		0=m(\e,\kappa,\s,\l)\bp \b_1\\ \b_e\ep=[-\l+M(\e,\kappa,\s)+\cE(\e,\kappa,\s)\l+\cO(\e^2,\e\s,\s^2)\cO(|\l|^2)]\bp \b_1\\ \b_2\ep,
	\end{split}
	\end{equation}
	where
		\ba\label{eq:bigM}
		M(\e,\kappa,\s)&:=2\e^2\a^2 \bp \Re\g & 0 \\ \Im\g & 0 \ep+\s^2[[\frac{1}{2}k_*^2\tl_{kk}(k_*,0)]]+\e\s i[[2i\kappa k_*\tl_{kk}(k_*,0)]]\\
		&\quad +\cO(\e^3,\e^2\s,\e\s^2,\s^3)
		\ea
	and $||\cE(\e,\kappa,\s)||\leq C||M(\e,\kappa,\s)||=\cO(\e^2,\e\s,\s^2)$ for some constant $C$. Here $\cE(\e,\kappa,\s )$ comes from the higher order terms in \eqref{eq:GStableRedEqnV2}. The precise form of $\cE$ is unneeded, but it is in principle computable by a more careful analysis of the $\cO(\l)$ terms in lemmas \eqref{lem:GTaylorOEE},\eqref{lem:GTaylorOES} and \eqref{lem:GTaylorOSS}. 

	Observe from Theorem \eqref{thm:copstab} that at $\s=0$, we know that $\det(m(\e,\kappa,0,0))=0$ for all $\e,\kappa$. We compute 
	\begin{equation}
		\frac{\d}{\d\l}\det(m(\e,\kappa,\s,\l))=-[M_{11}(\e,\kappa,\s)+M_{22}(\e,\kappa,\s)+\cO(\e^4,\e^3\s,\e^2\s^2,\e\s^3,\s^4) ]+\cO(\l),
	\end{equation}
	where $M_{ij}$ refers to the $ij$-th entry of the matrix $M(\e,\kappa,\s)$ defined in \eqref{eq:bigM}. 
	In particular, at $\s=0$, we see that $\frac{\d}{\d\l}\det(m(\e,\kappa,0,\l))|_{\l=0}=2\a^2\e^2\Re\g+\cO(\e^3)$. 

	To avoid pathological behavior at $\e=0$, we rescale $\l=\e^2\hal$ and $\s=\e\Hs$. Dividing out $\e^2$ in \eqref{eq:GStableRedEqnV2} and redoing the above calculation gives
	\begin{equation*}
		\frac{\d}{\d\hal}\det\frac{m(\e,\kappa,0,\l)}{\e^2}=\cO(1).
	\end{equation*}
	From here, we may safely apply the implicit function theorem to conclude $\hal=\hal(\e,\kappa,\Hs)$ with $\hal(\e,\kappa,0)\equiv0$. Undoing the scaling, we see that $\l$ admits the expansion
	\begin{equation}\label{eq:lambdaexp}
		\l(\e,\kappa,\s)=c_1(\e,\kappa)\s+c_2(\e,\kappa)\s^2+\cO(\s^3),
	\end{equation}
	where $c_1,c_2$ smoothly depend on $\e,\kappa$, moreover, they admit a smooth extension to $\e=0$. Because we've forced $\s\sim\e$ and $\l\sim\e^2$, the error term $\cE(\e,\kappa,\s)\l\sim\e^4$ and so it's a negligible error term. 
	Letting $\Lambda(\e,\kappa,\s)$ be the eigenvalue of $M(\e,\kappa,\s)$ satisfying $\L(\e,\kappa,0)\equiv 0$, then we morally should have $\l(\e,\kappa,\s)\approx\L(\e,\kappa,\s)$. Before we compute $\l$, 
	we show that $\L$ matches the prediction of (cGL).

	\begin{proposition}\label{eq:Lambda}
		Let $\L(\e,\kappa,\s)=C_1(\e,\kappa)\s+C_2(\e,\kappa)\s^2+\cO(\s^3)$ be an eigenvalue of $M(\e,\kappa,\s)$. Then we have
		\begin{equation}\label{eq:bigC1}
			C_1(\e,\kappa)=2i\kappa k_*\e\left(\frac{\Im\g\Re\tl_{kk}(k_*,0)}{\Re\g}-\Im\tl_{kk}(k_*,0) \right)+\cO(\e^2)
		\end{equation}
		and
		\ba\label{eq:bigC2}
			C_2(\e,\kappa)&=\frac{2\kappa^2 k_*^2}{\a^2\Re\g}\big((\Re\tl_{kk}(k_*,0))^2+\frac{\Re\tl_{kk}(k_*,0)\Im\tl_{kk}(k_*,0)\Im\g } {\Re\g}+\frac{\Im\g}{\Re\g}\big(\Im\tl_{kk}(k_*,0)\\
			&\quad +\frac{\Im\g\Re\tl_{kk}(k_*,0) }{\Re\g} \big) \big)+\cO(\e).
			\ea
		That is, $\L$ agrees with the complex Ginzburg-Landau approximation to the appropriate lowest order in $\e$.
	\end{proposition}
	\begin{proof}
		First, let $L,R$ be the vectors spanning the left/right kernel of $M(\e,\kappa,0)$:
		\begin{equation}
			L=\bp -\frac{\Im\g}{\Re\g} & 1 \ep+\cO(\e) \quad\quad R=\bp 0 \\ 1 \ep+\cO(\e),
		\end{equation}
		and $L^\perp, R^\perp$ be the left/right eigenvectors associated to the nonzero eigenvalue of $M(\e,\kappa,0)$ given by
		\begin{equation}
			L^\perp=\bp 1 & 0 \ep+\cO(\e) \quad \quad R^\perp=\bp 1\\ \frac {\Im\g} {\Re\g} \ep+\cO(\e).
		\end{equation}
		From here, it is a straightforward computation using the spectral identities with $\e,\kappa$ fixed to evaluate $C_1,C_2$.
	\end{proof}

	It is important that the $\cO(\s)$ term have no real part for either $\l$ or $\L$.

	\begin{lemma}\label{lem:Re0}
		$\Re c_1(\e,\kappa)=\Re C_1(\e,\kappa)=0$ for all $\e,\kappa$.
	\end{lemma}

	\begin{proof}
		Let $RW(\xi):=\bar{W}(\xi)$. Then we have the following claim.

		\begin{claim}
			For all $\e,\kappa,\s,\l$ and all $W$,
			\begin{equation}\label{eq:RSym}
				B(\e,\kappa,\s,\l)RW=RB(\e,\kappa,-\s,\bar{\l})W.
			\end{equation}
		\end{claim}

		To prove the claim, recall that $B(\e,\kappa,\s,\l)=L(k,\mu;\s)+dk\d_\xi+iC(\e,\kappa)\s+D\cN(\tilde{u}_{\e,\kappa})-\l$ for some real constant $C(\e,\kappa)$. In particular, the only nontrivial part of the claim is to show that $RL(k,\mu;\s)W=L(k,\mu;-\s)RW$. This can be done as follows.
		Assuming that $W$ is real valued,
		\ba
		RL(k,\mu;\s)W&=\overline{\sum_{\eta\in\ZZ}S(k(\eta+\s),\mu)\hat{W}(\eta)e^{i\eta\xi}}\\
		&=
	\sum_{\eta\in\ZZ}S(k(-\eta-\s),\mu)\hat{W}(-\eta)e^{-i\eta\xi}=L(k,\mu;-\s)W,
			\ea
		where we've used the reality condition on $L(k,\mu)$ and $W$. Splitting $W=W_R+iW_I$ where $W_R,W_I$ are real valued, we get by applying the above identity and linearity
		\ba
			RL(k,\mu;\s)W&=RL(k,\mu;\s)W_R-iRL(k,\mu;\s)W_I\\
			&=
			L(k,\mu;-\s)W_R-iL(k,\mu;-\s)W_I=L(k,\mu;-\s)(W_R-iW_I)=L(k,\mu;-\s)RW.
		\ea

		Following \cite{SZJV}, we 
		observe that the reduced equation has the same symmetry, and since $M(\e,\kappa,\s)=m(\e,\kappa,\s,0)$ it follows that $M$ inherits the symmetry as well. In particular for all $\b_1,\b_2$, we have
		\begin{equation*}
			\bar{m}(\e,\kappa,-\s,\bar{\l})\bp \b_1 \\ \b_2\ep=m(\e,\kappa,\s,\l)\bp \b_1\\ \b_2\ep
		\end{equation*}
		and
		\begin{equation*}
			\bar{M}(\e,\kappa,-\s)\bp \b_1\\ \b_2\ep=M(\e,\kappa,\s)\bp \b_1\\ \b_2\ep.
		\end{equation*}
		From these identities we conclude that $\bar{\l}(\e,\kappa,-\s)=\l(\e,\kappa,\s)$ 
		and $\bar{\L}(\e,\kappa,-\s)=\L(\e,\kappa,\s)$, whence Taylor expanding with respect to $\s$
		then proves the lemma.
	\end{proof}

	\begin{remark}
		In $O(2)$ invariant systems, e.g. the Brusselator model of \cite{SZJV}, 
		one has an extra symmetry $R_1W(\xi):=W(-\xi)$. This extra symmetry can be used to show that the matrix entries of $m$ are either even or odd with respect to $\s$ depending on whether they are diagonal or off-diagonal respectively. The symmetry we've used here is referred to as $R_2$ 
		in \cite{SZJV}. Notice that the $R_2$ symmetry is the only one that can be guaranteed for all systems we are considering here because it is a manifestation of the assumption that $L(k,\mu)$ maps real functions to real functions.
	\end{remark}

	\begin{theorem}\label{thm:specagree}
		There holds
		$\l(\e,\kappa,\s)=\L(\e,\kappa,\s)+\cO(\e^4,\e^3\s,\e^2\s^2,\e\s^3,\s^4 )$, hence the predictions of complex Ginzburg-Landau hold to lowest order.
	\end{theorem}

	\begin{proof}
		Observe that \eqref{eq:GStableRedEqnV2} is equivalent to
		\begin{equation}
			((I-\cE(\e,\kappa,\s))\l+\cO(\e^2|\l|^2,\e\s|\l|^2,\s^2|\l|^2))\bp \b_1\\ \b_2\ep=M(\e,\kappa,\s)\bp \b_1\\ \b_2\ep
		\end{equation}
		for some nonzero $\b_1+i\b_2$. Since $\cE(\e,\kappa,\s)=\cO(\e^2,\e\s,\s^2)$, it follows that $I-\cE(\e,\kappa,\s)$ is invertible for small $\e,\s$. So, we see that $\l$ is an eigenvalue of $(I-\cE(\e,\kappa,\s))^{-1}M(\e,\kappa,\s)$ up to quadratic errors in $\l$. We can expand $(I-\cE)^{-1}$ into its Neumann series to find
		\begin{equation}\label{eq:NeumannSeries}
			(\l+\cO(\e^2|\l|^2,\e\s|\l|^2,\s^2|\l|^2))\bp \b_1\\ \b_2\ep=\left(\sum_{n=0}^\infty\cE(\e,\kappa,\s)^nM(\e,\kappa,\s)  \right)\bp \b_1\\ \b_2\ep.
		\end{equation}
		But we also have that both $\cE(\e,\kappa,\s),M(\e,\kappa,\s)$ are $\cO(\e^2,\e\s,\s^2)$, so 
		$$
		\cE(\e,\kappa,\s)M(\e,\kappa,\s)=\cO(\e^4,\e^3\s,\e^2\s^2,\e\s^3,\s^4).
		$$
		A priori we know that $\l=\cO(\s)$, so in particular the quadratic errors in \eqref{eq:NeumannSeries} are at least as small as the error bound given by $\cE(\e,\kappa,\s)M(\e,\kappa,\s)$. Hence, we have by matching orders of $\e,\s$ that $\l(\e,\kappa,\s)=\L(\e,\kappa,\s)+\cO(\e^4,\e^3\s,\e^2\s^2,\e\s^3,\s^4)$.
	\end{proof}

	\begin{remark}
		Somewhat surprisingly, we have quartic agreement between $\l$ and $\L$. From this we conclude that $\cE(\e,\kappa,\s)\l$ is essentially a negligible error term. The usefulness of theorem \eqref{thm:specagree} is unfortunately somewhat tempered by the observation that $\L$ is not known to high enough order to make full use of the result.

		The argument presented here is closely related to the argument presented in 
		\cite{SZJV}, the main difference being that we here perform the Weierstrass preparation step of
		\cite{SZJV}``manually'' and more importantly we replace the quadratic formula approach with a spectral perturbation theoretic argument. This is important because if one starts to add conservation laws to the system, as
		described in the open problem portion of Section \ref{s:disc}, 
		then the size of the matrix $m(\e,\kappa,\s,\l)$ increases and we lose explicit formulas to work with.
	\end{remark}

	In order to show how theorem \eqref{thm:specagree} implies the general stability result, we split $\Hs$ into 
	regions 
	\begin{enumerate}
		\item $|\Hs|\leq\frac{1}{C}$
		\item $\frac{1}{C}\leq|\Hs|\leq C$
		\item $|\Hs|\geq C$
	\end{enumerate}
	depending on a large fixed constant $C$ and work with the rescaled quantities $\l=\e^2\hal$, $\s=\e\Hs$

	\begin{theorem}\label{thm:genstab}[General Stability]
		For any $\nu_0>0$, there exists an $\e_0>0$ so that for all $\e<\e_0$ and $|\kappa|\leq \kappa_E$, $\tilde{u}_{\e\kappa}$ is linearly stable 
		if
		$\kappa^2\leq (1-\nu_0)\kappa_S^2$,
		and linearly unstable if
		$\kappa^2\geq(1+\nu_0)\kappa_S^2$, 
		where $\kappa_E$ and $\kappa_S$ are as defined in Section \ref{sec:cglstability}. 
	\end{theorem}

	\begin{proof}
		In region 1, where $|\Hs|\leq\frac{1}{C}$, we use the expansion $\Re\hal(\e,\kappa,\s)=C_2(\e,\kappa)\Hs^2+\cO(\e,|\Hs|^3)$ by lemma \eqref{lem:Re0}. Since $\e$ and $\Hs$ are small we conclude that $\Re\hal(\e,\kappa,\Hs)<0$ whenever $\kappa^2\leq (1-O(\e))\kappa_S^2$ and $\Re\hal(\e,\kappa,\Hs)>0$ for $\kappa^2>(1+O(\e))\kappa_S^2$ and $\Hs\not=0$. 
		Note that the other eigenvalue in the reduced equation has negative real part when $\s=0$, and so it continues to have negative real part when $\Hs$ is small by continuity. Henceforth, we assume in addition that $\kappa^2\leq (1-\nu_0)\kappa_S^2$.

		In region 2, the unperturbed spectrum for complex Ginzburg-Landau is given by the eigenvalues of the matrix $\hat{M}(\e,\kappa,\Hs)$ defined by
		\begin{equation*}
			\hat{M}(\e,\kappa,\Hs):=\lim_{\e\rightarrow 0}\frac{1}{\e^2}M(\e,\kappa,\e\Hs)=2\a^2 \bp \Re\g & 0 \\ \Im\g & 0 \ep+\Hs^2[[\frac{1}{2}k_*^2\tl_{kk}(k_*,0)]]+\Hs i[[2i\kappa k_*\tl_{kk}(k_*,0)]].
		\end{equation*}
		These eigenvalues have uniformly negative real part for $|\kappa|\leq \kappa_S$ and $|\Hs|\sim 1$. The perturbed spectrum is therefore also negative for $\eps$ sufficiently small, as it differs by an $\cO(\e)$ term uniformly in the compact interval of $\frac{1}{C}\leq |\Hs|\leq C$.

		To complete the argument, we look at region 3. Because we've taken $|\s|\ll 1$ at the beginning of this section, it follows that $|\Hs|\gg 1$ with $|\e,\s|\ll 1$ is equivalent to $|\Hs|\gg 1$ with the higher order terms still negligible. In this regime, the eigenvalues of $\hat{m}$, where $\hat{m}$ is defined by $m=\e^2\hat{m}$, are equal to eigenvalues of $\frac{1}{2}\Hs^2[[\tl_{kk}(k_*,0)]]$ up to higher order terms, and a quick calculation shows that the eigenvalues of $[[z]]$ for $z\in\CC$ are given by $z,\bar{z}$.
	\end{proof}


	\section{Other Nonlinearities}\label{sec:other}
	In \cite{WZ},
	it was observed that any nonlinearity $\sN:H^s_{per}(\RR;\RR^n)\rightarrow L^2_{per}(\RR;\RR^n)$ satisfying the following could be used in the Lyapunov-Schmidt reduction.
	\begin{hypothesis} The nonlinear function $\sN$ satisfies:
		\begin{itemize}
			\item $\sN$ is a smooth map.
			\item $\sN$ is translation invariant in the sense that $\tau_y\sN(u)(x)=\sN(\tau_yu)(x)$ for all $x,y\in\RR$ and $u\in H^s_{per}(\RR;\RR^n)$ where $\tau_hf(x):=f(x-h)$ is a translation.
			\item $\sN(0)=D\sN(0)=0$.
		\end{itemize}
	\end{hypothesis}
	From these hypotheses, we can see that such a map $\sN$ sends the space of $\frac{2\pi}{k}$ periodic $H^s$ functions to periodic $L^2$ functions of the same period. We have the isomorphism 
	$$
	I_k:H^s_{per}([0,\frac{2\pi}{k}];\RR^n )\rightarrow H^s_{per}([0,2\pi];\RR^n )
	$$
	given by $I_ku(x):=u(kx)=u(\xi)$ and $H^s_{per}([0,X];\RR^n)$ is the space of $X$-periodic $H^s_{loc}$ functions. In a slight abuse of notation, $I_k$ will denote the same isomorphism for all $s$. Hence we can identify $\sN$ defined on $H^s_{per}$ with a map $\tilde{\sN}$ from $H^s(S^1_1;\RR^n)\times(0,\infty)$ defined by $I_k\sN(I_k^{-1}u)(x)=\tilde{\sN}(u,k)(\xi)$ whenever $u$ is $2\pi$-periodic. In another abuse of notation, we will drop the tilde on $\tilde{\sN}$.

	\begin{remarks}
			1. In the case of a local nonlinearity, e.g. the Burgers type nonlinearity 
				$$
				\cN(u)=\cQ(u,\d_x u)
				$$
				for some fixed bilinear form $\cQ:\RR^n\times\RR^n\rightarrow\RR^n$ we have that $\cN(U,k)=k\cQ(U,\d_\xi U)$, which follows from $\d_x=k\d_\xi$. In effect, all this formalism does is replace $\d_x$ by $k\d_\xi$. For a nonlocal example, fix a sufficiently rapidly decaying $\phi\in W^{1,1}(\RR)$ and a bilinear form $\cQ:\RR^n\times\RR^n\rightarrow\RR^n$, and consider the Keller-Segel \cite{BBTW} 
				type nonlinearity $\cN_\phi$ defined on $H^2_{per}$ functions by
			\begin{equation}
			\cN_\phi(u)(x):=\d_x(\cQ(u,\d_x(u*\phi)))(x)
			\end{equation}
			where $*$ denotes convolution on $\RR$. Fix a $2\pi$-periodic $U=U(\xi)$ and $k>0$, then we have that
			\begin{equation*}
			I_k(d_x\phi*(I_k^{-1}U)(x))=I_k(d_x\int_{\RR}\phi(y)U(x-y)dy)=k\d_\xi(\phi*U)(\xi)
			\end{equation*}
			for $\xi=kx$.
			
			2. Morally, one wants to think of the map $H^s_{per}([0,2\pi];\RR^n)\times(0,\infty)\rightarrow H^s_{per}(\RR;\RR^n)$ given by $(u,k)\rightarrow I_ku$ as a homeomorphism with ``inverse'' $u\rightarrow (\overline{u},k)$ where $\frac{2\pi}{k}$ is the minimal period of $u$ and $\overline{u}=I_ku$. However, while $(u,k)\rightarrow I_ku$ is a continuous surjection, it dramatically fails to be injective. Equally troubling is that the proposed inverse map is only defined for nonconstant functions and fails to be continuous.
	\end{remarks}

	From the Schwartz kernel theorem \cite{H},
	we see that $D_u\sN(\overline{u},k)v$ admits the representation
	\begin{equation}\label{eq:SKT}
		D_u\sN(\overline{u},k)v(\xi)=\sum_{\eta\in\ZZ}\cK(\xi,k\eta;\overline{u},k)\hat{v}(\eta)e^{i\eta\xi}.
	\end{equation}
	We can compute $\cK(\xi,\eta;u,k)$ in terms of $D_U\sN(u,k)$ using the identity
	\begin{equation}\label{eq:SKTKeyId}
		\cK(\xi,k\eta;\overline{u},k)=e^{-i\eta\xi}D_u\sN(\overline{u},k)e^{i\xi\eta}.
	\end{equation}
	In particular, our assumption that $\sN$ be smooth with respect to $k$ implies that $\cK$ depends
	smoothly on $k$.

	A primary use of the Schwartz kernel representation is to compute $D\sN(u,k)e^{i\s\xi}v(\xi)$ in terms of $\cK(\xi,k\eta;u,k)$, as
	\begin{equation}
		D_u\sN(\overline{u},k)e^{i\s\xi}v(\xi)=\sum_{\eta\in\ZZ}\cK(\xi,k\eta;\overline{u},k)\hat{v}(\eta-\s)e^{i\eta\xi}=\sum_{\tilde{\eta}\in\ZZ}e^{i\s\xi}\cK(\eta,k(\tilde{\eta}+\s);\overline{u},k)\hat{v}(\tilde{\eta})e^{i\tilde{\eta}\xi}.
	\end{equation}
	Before we show the details of the derivation of the reduced equation in the following examples, we compute the constant $\g$ in the general case.

	\begin{lemma}\label{lem:GeneralGamma}
		The constant $\g$ is given by
		\ba\label{eq:generalgamma}
		\g&=\ell\big[\sQ(0,1)(-\frac{1}{2}S(0,0)^{-1}\Re\sQ(1,-1)(r,\bar{r}),r)
		+\sQ(2,-1)(-\frac{1}{4}S(2k_*,0)^{-1}\sQ(1,1)(r,r),\bar{r})\\
		&\quad +\frac{1}{16}\sC(1,1,-1)(r,r,\bar{r}) \big],
		\ea
		where, informally, we've identified $\sQ(nk_*,mk_*)$ with $\sQ(n,m)$ and similarly for $\sC$.
	\end{lemma}
	\begin{proof}
		We begin with the following observation.
		\begin{obs}
			$\d_k^j\sN(0,k)=\d_k^jD_u\sN(0,k)=0$ for all $k>0$ and all $j\in\NN$.
		\end{obs}
		This follows from $\sN(0,k)=I_k\sN(I_k^{-1}0)\equiv0$ and $D_u\sN(\overline{u},k)=I_kD_u\sN(I_k^{-1}\overline{u})I_k^{-1}$ which is also identically zero when $\overline{u}=0$.

		We Taylor expand the nonlinearity, and upon applying the above observation, discover that
		\begin{equation}
			\sN(U^\e,k)=\frac{1}{2}D^2_u\sN(0,k_*)(U^\e,U^\e)+\frac{1}{6}D^3_u\sN(0,k_*)(U^\e,U^\e,U^\e)+\frac{1}{2}\kappa\d_k D_u^2\sN(0,k_*)(U^\e,U^\e)+\cO(\e^4).
		\end{equation}
		Since each form in the above is translation invariant, it follows that each is a multilinear Fourier multiplier operator, which we will denote by
		\begin{equation}
				D_u^2\sN(0,k)(U,V)=\sum_{\eta_1,\eta_2\in\ZZ}\sQ(k\eta_1,k\eta_2)(\hat{U}(\eta_1),\hat{V}(\eta_2))e^{i\xi(\eta_1+\eta_2)}
		\ee
				and
		\be
				D_u^3\sN(0,k)(U,V,W)=\sum_{\eta_1,\eta_2,\eta_3\in\ZZ}\sC(k\eta_1,k\eta_2,k\eta_3)(\hat{U}(\eta_1),\hat{V}(\eta_2),\hat{W}(\eta_3))e^{i\xi(\eta_1+\eta_2+\eta_3)}.
		\end{equation}

		Writing $\kappa=\e\kappa$, we find that 
		\begin{equation}
			\kappa\d_k D_u^2\sN(0,k_*)(U^\e,U^\e)=\e \tilde{\sQ}(\d_{\Hx} U^\e, U^\e)
		\end{equation} 
		for some known bilinear form $\tilde{\sQ}$. At $\cO(\e^2)$, the relevant terms of which are given by
		\begin{equation}\label{eq:GenPsi}
			\Psi_2(\Hx,\Ht)=-\frac{1}{4}A(\Hx,\Ht)^2S(2k_*,0)^{-1}\sQ(k_*,k_*)(r,r)
			\ee
			and
			\be
			\Psi_0(\Hx,\Ht)=-\frac{1}{4}|A(\Hx,\Ht)|^2S(0,0)^{-1}\left[\sQ(k_*,-k_*)(r,\bar{r})+\sQ(-k_*,k_*)(\bar{r},r) \right].
		\end{equation}
		From this observation, we may conclude that the nonlinearity contributes at $\cO(\e^3)$ 
		and Fourier mode $e^{i\xi}$ the term
		\begin{equation}
				D^2_u\sN(0,k_*)(\Psi_0,Ar)+D^2_u\sN(0,k_*)(\Psi_2,\bar{Ar})+\frac{1}{16}D^3_u\sN(0,k_*)(Ar,Ar,\bar{Ar}).
		\end{equation}
		Plugging in \eqref{eq:GenPsi} and applying $\ell$ gives the desired formula.
	\end{proof}

	\begin{remark}
		An alternative argument that $\d_k^j\sN(0,k)\equiv 0$ is the observation that $\sN(\overline{u},k)$ is independent of $k$ for all constant functions $\overline{u}$.
	\end{remark}

	\subsection{Two local examples}
	In this section, we will focus on the following model reaction-diffusion systems for $\RR^n$ valued $u$:
	\be\label{eq:ModelLocalSystems} 
		u_t=D(\mu)\d_x^2u+A(\mu)u+\frac{1}{2}\d_x^2(u\circ u)=L(\mu)u+u\circ u_{xx}+u_x\circ u_x 
		\ee
		and
		\be\label{eq:model2}
		u_t=D(\mu)\d_x^2u+A(\mu)u+\frac{1}{2}\d_x(u\circ u)=L(\mu)u+u\circ u_x,
		\ee
	where $u\circ u$ denotes the Hadamard product on $\RR^n$. For simplicity, we take $\mu=\e^2$.

	\begin{remark}
		For the purposes of these examples, the linear part being reaction-diffusion will suffice because we are primarily interested in the effects of derivatives on the nonlinear term in the stability calculation. 
		Assuming a reaction-diffusion equation forces $d(\e,\kappa)\equiv 0$ by $O(2)$ invariance and in addition forces $\ell$ and $r$ to be real vectors. The purpose of having two nonlinearities is that the vector form of Burgers equation turns out to be linear in $\s$, hence automatically drops out of $B_{\s\s}(\e,\kappa,\l,\s)$.
	\end{remark}

	It will suffice to show that the reduced equation agrees with the prediction of complex Ginzburg-Landau, because the spectral continuity argument primarily only required 
	$$
	||D_u\cN(\tilde{u}_{\e,\kappa},k)||_{H^s\rightarrow L^2}=\cO(\e). 
	$$
	Moreover, the refined error estimate for coperiodic stability carries through with minimal modifications because essentially only $u$-derivatives are taken.

	\begin{theorem}
		The reduced equation for the first system in \eqref{eq:ModelLocalSystems} agrees with the prediction of complex Ginzburg-Landau.
	\end{theorem}
	The proof of this theorem will take place over the space of several lemmas, 
	the first of which computes the relevant derivatives of the Bloch operator for this system.
	\begin{lemma}\label{lem:Example1DBloch}
		The Bloch operator for this system: 
		\begin{equation}\label{eq:BlochLoc1}
		\begin{split}
		B(\e,\kappa,\l,\s)v=&D(\e^2)k^2(\d_\xi+i\s)^2v+A(\e^2)v+\\
		&+k^2 \tilde{u}_{\e,\kappa}\circ (\d_\xi+i\s)^2v+(\d_x^2\tilde{u}_{\e,\kappa})\circ v+2k(\d_x\tilde{u}_{\e,\kappa})\circ (\d_\xi+i\s)v-\l v,
		\end{split}
		\end{equation}
		has the derivatives:
		\begin{align}\label{eq:DBlochLoc1}
			B(0,\kappa,\l,0)v&=k_*^2D(0)\d_\xi^2v+A(0)v-\l v\\
			B_\e(0,\kappa,\l,0)v&=2\kappa k_* D(0)\d_\xi^2v+k_*^2\Upsilon_\a\circ \d_\xi^2v+\d_x^2\Upsilon_\a\circ v+2k_* \d_x\Upsilon_\a\circ \d_\xi v \\
			B_\s(0,\kappa,\l,0)v&=2ik_*^2D(0)\d_\xi v\\
			\begin{split}
			B_{\e\e}(0,\kappa,\l,0)v&=2\kappa^2 D(0)\d_\xi^2v+2k_*^2D_\mu(0)\d_\xi^2v+2A_\mu(0)v+2\kappa k_*\Upsilon_\a\circ \d_\xi^2v+\\
				&+k_*^2(\d_\e^2\tilde{u}_{0,\kappa})\circ\d_\xi^2v+(\d_\e^2\d_x^2\tilde{u}_{0,\kappa})\circ v+2\kappa (\d_x \Upsilon_\a)\circ v+2k_*(\d_\e^2\d_x \tilde{u}_{0,\kappa})\circ \d_\xi v
			\end{split}\\
			B_{\e\s}(0,\kappa,\l,0)v&=4i\kappa k_*D(0)\d_\xi v+2ik_*^2\Upsilon_\a\circ \d_\xi v+2ik_*\d_x\Upsilon_\a\circ v\\
			B_{\s\s}(0,\kappa,\l,0)v&=-2k_*^2D(0)v.
		\end{align}
	\end{lemma}

	\begin{proof}
		We first compute the terms coming from $L(k,\mu;\s)=k^2D(\mu)(\d_\xi+i\s)^2+A(\mu)$:
		\begin{align*}
			\d_\e L(k,\mu;\s)&=[2\kappa kD(\mu)+\mu_\e(\e)k^2D_\mu(\mu) ](\d_\xi+i\s)^2+A_\mu(\mu)\mu_\e(\e)\\
			\d_\s L(k,\mu;\s)&=2ik^2D(\mu)(\d_\xi+i\s)\\
			\begin{split}
			\d_\e^2 L(k,\mu;\s)&=[2\kappa^2D(\mu)+4\kappa\mu_\e(\e)D_\mu(\mu)+\mu_{\e\e}(\e)k^2D_\mu(\mu)+\mu_\e(\e)^2D_{\mu\mu}(\mu) ](\d_\xi+i\s)^2+\\ &+A_\mu(\mu)\mu_{\e\e}(\e)+A_{\mu\mu}(\mu)\mu_\e(\e)^2
			\end{split} \\
			\d_\e\d_\s L(k,\mu;\s)&=2i[2\kappa kD(\mu)+\mu_\e(\e)k^2D_\mu(\mu) ](\d_\xi+i\s)\\
			\d_\s^2 L(k,\mu;\s)&=-2k^2D(\mu).
		\end{align*}
		Taking $\mu(\e)=\e^2$ and then evaluating at $\e=\s=0$ gives the contribution from $L(k,\mu;\s)$ in \eqref{eq:DBlochLoc1}. 
		
		Now we consider the contribution of the linearization of the nonlinearity. 
		    We define 
		\begin{equation}
			D_u\cN(u;k,\s)v:=k^2 u\circ (\d_\xi+i\s)^2v+(\d_x^2u)\circ v+2k(\d_xu)\circ (\d_\xi+i\s)v.
		\end{equation}
		For the first derivatives with respect to $\e,\s$ we have
		\begin{align*}
		\begin{split}
			\d_\e D_u\cN(\tilde{u}_{\e,\kappa};k,\s)v&=2\kappa k \tilde{u}_{\e,\kappa}\circ (\d_\xi+i\s)^2v+2\kappa(\d_x\tilde{u}_{\e,\kappa})\circ (\d_\xi+i\s)v+\\&
			+k^2 (\Upsilon_\a+\cO(\e))\circ (\d_\xi+i\s)^2v+\d_x^2(\Upsilon_\a+\cO(\e))\circ v+2k\d_x(\Upsilon_\a+\cO(\e))\circ (\d_\xi+i\s)v
		\end{split}\\
			\d_\s D_u\cN(\tilde{u}_{\e,\kappa};k,\s)v&=2i k^2 \tilde{u}_{\e,\kappa}\circ (\d_\xi+i\s)v+2ik(\d_x\tilde{u}_{\e,\kappa})\circ v,
		\end{align*}
		where we've used $\tilde{u}_{\e,\kappa}=\e\Upsilon_\a+\e V(\e,\kappa,\a)$ with $V(\e,\kappa,\a)=\cO(\e)$. Taking $\e=\s=0$ as before gives the desired identities. The second derivatives follow from a similar, though lengthier, computation.
	\end{proof}

	In this example, Proposition \eqref{prop:CVExist} continues to hold because the proof relied entirely on the structure of $L(k,\mu)$ and the smoothness of $\cN(u)$. For convenience, we recall the necessary identities from proposition \eqref{prop:CVExist}:
	\begin{enumerate}
		\item $\cV(0,\kappa,\l,0,\b)\equiv0$
		\item $\cV_\e(0,\kappa,\l,0,\b)=-T_\l (I-P)B_\e(0,\kappa,\l,0)\Upsilon_\b$
		\item $\cV_\s(0,\kappa,\l,0,\b)=-T_\l (I-P)B_\s(0,\kappa,\l,0)\Upsilon_\b$,
	\end{enumerate}
	where $T_\l=\left[(I-P)B(0,\kappa,\l)(I-P)\right]^{-1}$. 
	As before we have the key commutation relationship $PB(0,\kappa,\l,0)=-\l P$. Following the procedure of section \eqref{sec:gencase}, we simplify each term of the Taylor expansion of $PB(\e,\kappa,\l,\s)(\Upsilon_\b+\cV(\e,\kappa,\l,\s,\b))=0$ in its own lemma. Similarly as in Section  \eqref{sec:gencase},
the key commutation relationship implies that for each term, 
the corresponding derivative of $\cV$ does not contribute; e.g. in the $\cO(\s)$ term $\cV_\s$ is annihilated by $P$.

	\begin{lemma}\label{lem:Example1O1}
		$PB(0,\kappa,\l,0)\Upsilon_\b$ in reduced form is given by
		$
			-\l \Upsilon_\b.
		$
	\end{lemma}
	\begin{proof}
		The proof is identical to the one in Lemma \eqref{lem:GTaylorO1}.
	\end{proof}

	\begin{lemma}\label{lem:Example1OE}
		$PB_\e(0,\kappa,\l,0)\Upsilon_\b+PB(0,\kappa,\l,0)\cV_\e(0,\kappa,\l,0,\b)$ vanishes identically.
	\end{lemma}

	\begin{proof}
		Expanding this using the appropriate derivative in Lemma \eqref{lem:Example1DBloch}, we get
		\begin{equation}
			P\left(2\kappa k_* D(0)\d_\xi^2\Upsilon_\b+k_*^2\Upsilon_\a\circ \d_\xi^2\Upsilon_\b+\d_x^2\Upsilon_\a\circ \Upsilon_\b+2k_* \d_x\Upsilon_\a\circ \d_\xi \Upsilon_\b \right)= 2\kappa k_* PD(0)\d_\xi^2\Upsilon_\b,
		\end{equation}
		because $\d_x^N\Upsilon_\a\circ \d_\xi^M\Upsilon_\b$ is Fourier supported in $\{0,\pm 2 \}$ for any $N,M\in\NN$ and hence in the kernel of $P$. 
		Since we have an $O(2)$ invariant linear operator, the dispersion relation $\tl(k,\mu)$ is real 
		in a neighborhood of $(k_*,0)$. This implies that $2k_*\ell D(0)r=\tl_k(k_*,0)=0$.
	\end{proof}

	\begin{lemma}\label{lem:Example1OS}
		$PB_\s(0,\kappa,\l,0)\Upsilon_\b+PB(0,\kappa,\l,0)\cV_\s(0,\kappa,\l,0,\b)$ also vanishes identically.
	\end{lemma}

	\begin{proof}
		This follows from $B_\s(0,\kappa,\l,0)=2ik_*^2D(0)\d_\xi$ and a similar calculation to the one provided in the previous lemma.
	\end{proof}

	\begin{lemma}\label{lem:Example1OEE}
		In reduced form, $PB_{\e\e}(0,\kappa,\l,0)\Upsilon_\b+2PB_\e(0,\kappa,\l,0)\cV_\e(0,\kappa,\l,0,\b)$ is given by the following expression
		$$
		2\a^2 \bp \Re\g & 0 \\ \Im\g & 0 \ep \bp \b_1\\ \b_2\ep +\bp \cO(\e,|\l|) & \cO(|\l|) \\ \cO(\e,|\l|) & \cO(|\l|)\ep \bp \b_1\\ \b_2\ep.
		$$
	\end{lemma}

	\begin{proof}
		First, we expand $PB_{\e\e}(0,\kappa,\l,0)\Upsilon_\b$ using the appropriate derivative in Lemma \eqref{lem:Example1DBloch}:
		\begin{equation}\label{eq:Ex1OEE1}
		\begin{split}
		PB_{\e\e}(0,\kappa,\l,0)\Upsilon_\b&=-2k_*^2PD_\mu(0)\Upsilon_\b+2PA_\mu(0)v\Upsilon_\b-Pk_*^2(\d_\e^2\tilde{u}_{0,\kappa})\circ\Upsilon_\b+\\
		&+P(\d_\e^2\d_x^2\tilde{u}_{0,\kappa})\circ \Upsilon_\b+2k_*P(\d_\e^2\d_x \tilde{u}_{0,\kappa})\circ \d_\xi\Upsilon_\b,
		\end{split}
		\end{equation}
		where we've used $\d_\xi^2\Upsilon_\b=-\Upsilon_\b$, $PD(0)\Upsilon_\b=0$, and $P\left(\d_x^N\Upsilon_\a\circ\d_\xi^M\Upsilon_\b \right)=0$ as in the calculation for the $\cO(\e)$ term. 
		From spectral perturbation theory, we know that $\ell (-k_*^2D_\mu(0)+A_\mu(0))r=\tl_\mu(k_*,0)$. 
		To continue with this term, we need to compute $\d_\e^2\tilde{u}_{0,\kappa}$. To do this, we write down the equation that $\tilde{u}_{\e,\kappa}$ solves and differentiate with respect to $\e$ twice, to obtain
		\begin{equation}
		 	\d_\e^2 k^2D(\e^2)\d_\xi^2\tilde{u}_{\e,\kappa}+A(\e^2)\tilde{u}_{\e,\kappa}+k^2\d_\xi^2(\tilde{u}_{\e,\kappa}\circ \tilde{u}_{\e,\kappa})=0.
		\end{equation}

		Evaluating the derivative and collecting terms, we get
		\ba
			k^2D(\e^2)&\d_\xi^2\d_\e^2  \tilde{u}_{\e,\kappa}+A(\e^2)\d_\e^2\tilde{u}_{\e,\kappa}+2k^2\d_\xi^2(\tilde{u}_{\e,\kappa}\circ \d_\e^2\tilde{u}_{\e,\kappa})+2k^2\d_\xi^2(\d_\e\tilde{u}_{\e,\kappa}\circ\d_\e\tilde{u}_{\e,\kappa})\\
			&+2\big[(2\kappa kD(\e^2)+2\e k^2D_(\mu))\d_\xi^2+2\e A_\mu(\e^2) \big]\d_\e\tilde{u}_{\e,\kappa}+8\kappa k\d_\xi^2(\tilde{u}_{\e,\kappa}\circ \d_\e\tilde{u}_{\e,\kappa})\\
			&+\big[(2\kappa^2 D(\e^2)+2\e\kappa kD_\mu(\e^2)+2k^2D_\mu(\e^2)+4\e^2k^2D_{\mu\mu}(\e^2))\d_\xi^2+2A_\mu(\e^2)\\
			& +4\e^2A_{\mu\mu}(\e^2) \big]\tilde{u}_{\e,\kappa}
			+2\kappa^2\d_\xi^2(\tilde{u}_{\e,\kappa}\circ\tilde{u}_{\e,\kappa})=0.
			\ea
		Taking $\e=0$, we notice that the surviving terms are given by
		\begin{equation}
			k_*^2D(0) \d_\xi^2 \d_\e^2 \tilde{u}_{0,\kappa}+A(0)\d_\e^2\tilde{u}_{0,\kappa}+2k_*^2\d_\xi^2(\Upsilon_\a\circ\Upsilon_\a)+4\kappa k_*D(0)\d_\xi^2\Upsilon_\a=0.
		\end{equation}

		For our purposes, we need modes $0$ and $2$ of $\d_\e^2\tilde{u}_{0,\kappa}$. It is easily verified from the above equation that they are given by
		\begin{align}
			\widehat{\d_\e^2\tilde{u}_{0,\kappa}}(0)&=0,\\
			\widehat{\d_\e^2\tilde{u}_{0,\kappa}}(2)&=2k_*^2\a^2\left[-4k_*^2D(0)+A(0)\right]^{-1}(r\circ r).
		\end{align}
		
		The other term is $-2PB_\e(0,\kappa,\l,0)(I-P)T_\l(I-P)B_\e(0,\kappa,\l,0)\Upsilon_\b$, which by Taylor's theorem is given by $-2PB_\e(0,\kappa,\l,0)(I-P)T_0(I-P)B_\e(0,\kappa,\l,0)\Upsilon_\b+\cO(|\l|)\Upsilon_\b$. Thinking of $B_\e(0,\kappa,\l,0)$ as $\d_\e L(k,\mu;s)|_{\e=0}+\d_\e D_u\cN(\tilde{u}_{\e,\kappa};k,\s)|_{\e=0}$, we find that 
		$$
		-2PB_\e(0,\kappa,\l,0)(I-P)T_0(I-P)B_\e(0,\kappa,\l,0)\Upsilon_\b
		$$
		has four terms: ``linear-linear'', ``nonlinear-linear'', ``linear-nonlinear'', and ``nonlinear-nonlinear'' where ``linear'' corresponds to taking $\d_\e L(k,\mu;\s)|_{\e=0}$, ``nonlinear'' corresponds to 
$$
\d_\e D_u\cN(\tilde{u}_{\e,\kappa};k,\s)|_{\e=0},
$$
and the order determines which copy of $B_\e$ the term came from. 
		First we have the ``linear-linear'' term given by
		\ba
			P\big[&2\kappa k_* D(0)\d_\xi^2(I-P)T_0(I-P)2\kappa k_* D(0)\d_\xi^2  \big]\Upsilon_\b\\
			&=
			\kappa^2\Pi\big(2k_*D(0)(I-\Pi)\big[(I-\Pi)(k_*^2D(0)+A(0))(I-\Pi) \big]^{-1}(I-\Pi)2k_*D(0) \big)\Pi\b e^{i\xi}r\\
			&\quad  +c.c.
			\ea
		In terms of the symbol, the above reduces to
		\ba
			P\big[2\kappa k_* D(0)\d_\xi^2(I-P)&T_0(I-P)2\kappa k_* D(0)\d_\xi^2  \big]\Upsilon_\b=\\
			&\frac{1}{2}\kappa^2\Pi\big(S_k(k_*,0)\big[(I-\Pi)S(k_*,0)(I-\Pi) \big]^{-1}S_k(k_*,0) \big)\Pi\b e^{i\xi}r+c.c.
			\ea
		From the identity \eqref{eq:spectralid}, we see that in our special case this reduces to $\tl_{kk}(k_*,0)\Upsilon_\b$. 
		
		For the ``nonlinear-linear'' term, we have
		\begin{equation}
			2\kappa k_* P\left[k_*^2\Upsilon_\a\circ \d_\xi^2+\d_x^2\Upsilon_\a\circ+2k_* \d_x\Upsilon_\a\circ \d_\xi\right](I-P)T_0(I-P)D(0)\d_\xi^2\Upsilon_\b,
		\end{equation}
		which is zero because $\d_\xi^j(I-P)T_0(I-P)D(0)\d_\xi^2$ is a Fourier multiplier operator and hence 
		$$
		(I-P)T_0(I-P)D(0)\d_\xi^2\Upsilon_\b
		$$
		has Fourier support $\{\pm1 \}$, so that
		$\Upsilon_\a\circ \d_\xi^j(I-P)T_0(I-P)D(0)\d_\xi^2\Upsilon_\b$ is Fourier supported in $\{0,\pm2\}$. 
		For similar reasons, the ``linear-nonlinear'' term also vanishes.
	\end{proof}

	\begin{lemma}\label{lem:Example1OES}
	In reduced form, $PB_{\e\s}(0,\kappa,\l,0)\Upsilon_\b+PB_\e(0,\kappa,\l,0)\cV_\s(0,\kappa,\l,0,\b)+PB_\s(0,\kappa,\l,0)\cV_\e(0,\kappa,\l,0,\b)$ is given by
	\begin{equation}
	-i[[i\kappa k_*\tl_{kk}(k_*,0)]]\bp \b_1 \\ \b_2\ep+\cO(|\l|).
	\end{equation}
	\end{lemma}

	\begin{proof}
		There are three terms in the original expression in the lemma, given by $PB_{\e\s}(0,\kappa,\l,0)\Upsilon_\b$, $PB_\e(0,\kappa,\l,0)\cV_\s(0,\kappa,\l,0,\b)$, and $PB_\s(0,\kappa,\l,0)\cV_\e(0,\kappa,\l,0,\b)$. 
		First, looking at $PB_{\e\s}(0,\kappa,\l,0)\Upsilon_\b$ and expanding we get
		\begin{equation}
			PB_{\e\s}(0,\kappa,\l,0)\Upsilon_\b=P[4i\kappa k_*D(0)\d_\xi \Upsilon_\b+2ik_*^2\Upsilon_\a\circ \d_\xi \Upsilon_\b+2ik_*\d_x\Upsilon_\a\circ \Upsilon_\b].
		\end{equation}
		Note that this whole expression vanishes via reasoning similar to that which
		showed that the $\e$ coefficient identically zero. 

		For $PB_\e(0,\kappa,\l,0)\cV_\s(0,\kappa,\l,0,\b)$, we Taylor expand $T_\l=T_0+\cO(|\l|)$. Expanding using the formulas for $B$ and $\cV$, we get
		
		\begin{equation}
		\begin{split}
			PB_\e(0,\kappa,\l,0)\cV_\s(0,\kappa,\l,0,\b)=-P\left[2\kappa k_* D(0)\d_\xi^2+k_*^2\Upsilon_\a\circ \d_\xi^2+\d_x^2\Upsilon_\a\circ+2k_*\d_x\Upsilon_\a\circ \d_\xi \right]\\
			\cdot (I-P)T_0(I-P)[2ik_*^2D(0)\d_\xi]\Upsilon_\b+\cO(|\l|)\Upsilon_\b.
		\end{split}
		\end{equation}
		Similarly as in the ``linear-nonlinear'' calculation in lemma \eqref{lem:Example1OEE}, we can reduce the above to
		\ba
		PB_\e(0,\kappa,\l,0)\cV_\s(0,\kappa,\l,0,\b)&=-\kappa k_*P(2k_* D(0)\d_\xi^2)(I-P)T_0(I-P)(2ik_*D(0)\d_\xi)\Upsilon_\b\\
		&\quad +\cO(|\l|)\Upsilon_\b.
		\ea

		Notice that $P(2k_* D(0)\d_\xi^2)(I-P)T_0(I-P)(2ik_*D(0)\d_\xi)P$ is a Fourier multiplier operator with symbol supported on $\eta=\pm 1$:
		\begin{equation*}
			\Pi(-2k_* D(0))(I-\Pi)\left[(I-\Pi)[-k_*^2(\eta)^2D(0)+A(0) ](I-\Pi) \right]^{-1}(I-\Pi)(-2k_* D(0)\eta)\Pi.
		\end{equation*}
		That is, in terms of $S(k,\mu)$ and its derivatives, it is given by
		\begin{equation}
			\sgn(\eta)\Pi S_k(k_*,0)(I-\Pi)[(I-\Pi)S(\eta k_*,0)(I-\Pi)]^{-1}(I-\Pi)S_k(k_*,0)\Pi=-\frac{1}{2}\sgn(\eta)\tl_{kk}(k_*,0)\Pi.
		\end{equation}
		Hence, the contribution of $PB_\e(0,\kappa,\l,0)\cV_\s(0,\kappa,\l,0,\b)$ is given by
		\begin{equation}\label{eq:Example1PBECVS}
			\frac{1}{2}\kappa k_*\tl_{kk}(k_*,0)H\Upsilon_\b,
		\end{equation}
		where $H$ is the Hilbert transform. Analogously, we find that $PB_\s(0,\kappa,\l,0)\cV_\e(0,\kappa,\l,0,\b)$ contributes \eqref{eq:Example1PBECVS} as well. Applying the identity $\frac{1}{2}(z-\bar{z})=-i\frac{1}{2}(iz+\bar{iz})$, we can write the $\cO(\e\s)$ coefficient as
		\begin{equation}
			(-i[[i\kappa k_*\tl_{kk}(k_*,0) ]]+\cO(|\l|))\Upsilon_\b.
		\end{equation}
	\end{proof}

	\begin{lemma}\label{lem:Example1OSS}
			In reduced form, $\frac{1}{2}PB_{\s\s}(0,\kappa,\l,0)\Upsilon_\b+PB_\s(0,\kappa,\l,0)\cV_\s(0,\kappa,\l,0,\b)$ is given by
			\begin{equation}
			[[\frac{1}{2}k_*^2\tl_{kk}(k_*,0)]]\bp \b_1 \\ \b_2 \ep+\cO(|\l|)\bp \b_1 \\ \b_2\ep.
			\end{equation}
	\end{lemma}

	\begin{proof}
		There are two terms in this case, $\frac{1}{2}PB_{\s\s}(0,\kappa,\l,0)\Upsilon_\b$ and $PB_\s(0,\kappa,\l,0)\cV_\s(0,\kappa,\l,0,\b)$. For the first term, we use the formula for $B_{\s\s}$ to expand it into
		\begin{equation}
			\frac{1}{2}PB_{\s\s}(0,\kappa,\l,0)\Upsilon_\b=-2k_*^2PD(0)P\Upsilon_\b=0
		\end{equation}
		by expanding it into Fourier modes. 
		Writing out $PB_\s(0,\kappa,\l,0)\cV_\s(0,\kappa,\l,0,\b)$, we find that
		\begin{equation*}
			PB_\s(0,\kappa,\l,0)\cV_\s(0,\kappa,\l,0,\b)=-P(2ik_*^2D(0)\d_\xi)(I-P)T_\l (I-P)(2ik_*^2D(0)\d_\xi)\Upsilon_\b.
		\end{equation*}

		Taylor expanding with respect to $\l$, we get
		\begin{equation}
			PB_\s(0,\kappa,\l,0)\cV_\s(0,\kappa,\l,0,\b)=-P(2ik_*^2D(0)\d_\xi)(I-P)T_0 (I-P)(2ik_*^2D(0)\d_\xi)\Upsilon_\b+\cO(|\l|)\Upsilon_\b.
		\end{equation}
		Factoring out a $k_*^2$, we compute the symbol of the above operator to be
		\begin{equation}
			-k_*^2\Pi(2ik_*D(0)i\eta)[(I-\Pi)S(\eta k_*,0)(I-\Pi)]^{-1}(2ik_*D(0)i\eta)\Pi.
		\end{equation}
		Collecting the powers of $\eta$, and simplifying a bit, we get
		$$
			-\eta^2k_*^2\Pi S_k(k_*,0)[(I-\Pi)S(\eta k_*,0)(I-\Pi)]^{-1}S_k(k_*,0)\Pi.
			$$
		Since $\eta=\pm1$ and $S(k,\mu)$ is even with respect to $k$, an application of the spectral identity reduces the above, finally, to
		$
			\frac{1}{2}k_*^2\tl_{kk}(k_*,0)\Pi.
			$
	\end{proof}

	We now turn to the other local example given by \eqref{eq:model2}. However, we will take a slightly different approach than the first model. The first key observation is that Lemma \ref{lem:corefined} gives the reduced equation at $\s=0$, or equivalently, all the terms in the Taylor expansion of $PB(\Upsilon_a+\cV)=0$ that only have $\e$-derivatives. Hence, we only need to compute the terms in the Taylor expansion that have at least one $\s$-derivative on either $B$ or $\cV$. From the identity \eqref{eq:SKTKeyId} relating the Schwartz kernel $\cK(\xi,\nu;\tilde{u}_{\e,\kappa},k)$ to $D_U\cN(\tilde{u}_{\e,\kappa},k)$, we see that we can write the corresponding Bloch operator as
	\be
	B(\e,\kappa,\l,\s)=L(k,\mu;\s)+dk\d_\xi+iC(\e,\kappa)\s+e^{-i\s\xi}D_U\cN(\tilde{u}_{\e,\kappa},k)e^{i\s\xi}
	\ee
	\begin{remark}
		Since this example model has $\cN(U,k)=\frac{1}{2}\d_x\bp (U^1)^2\\ (U^2)^2\ep$ where $U=(U^1,U^2)$, $d$ and $C$ can be nonzero as the original system is not $O(2)$-invariant.
	\end{remark}
	For this specific model, we have that
	\ba
		e^{-i\s\xi}D_U(\tilde{u}_{\e,\kappa},k)e^{i\s\xi}V&=e^{-i\s\xi}\Big(k\bp \tilde{u}^1_{\e,\kappa} & 0\\ 0 & \tilde{u}^2_{\e,\kappa}\ep\d_\xi e^{i\s\xi}V+k\bp \d_\xi \tilde{u}^1_{\e,\kappa} & 0 \\ 0 & \d_\xi\tilde{u}^2_{\e,\kappa}\ep e^{i\s\xi}V\Big)=\\ 
		&\quad=k\bp \tilde{u}^1_{\e,\kappa} & 0\\ 0 & \tilde{u}^2_{\e,\kappa}\ep(\d_\xi+i\s)V+k\bp \d_\xi \tilde{u}^1_{\e,\kappa} & 0 \\ 0 & \d_\xi\tilde{u}^2_{\e,\kappa}\ep V
	\ea
	Where $\tilde{u}_{\e,\kappa}=(\tilde{u}^1_{\e,\kappa},\tilde{u}^2_{\e,\kappa})$. As the original nonlinearity $\cN(U,k)$ is given by $Q(U,U;k)$ for a bilinear and translation-invariant operator $Q$, we can write $Q$ as a bilinear Fourier multiplier \cite{Mu}. In particular, the multiplier can be computed using the convolution theorem as
	\be
		Q(U,V;k)=\sum_{\eta_1,\eta_2\in\ZZ}\frac{1}{2}ik(\eta_1+\eta_2)\bp \hat{U}^1(\eta_1)\hat{V}^1(\eta_2)\\\hat{U}^2(\eta_1)\hat{V}^2(\eta_2)\ep e^{i(\eta_1+\eta_2)\xi}
	\ee
	Where as before, $\hat{U}$ denote the Fourier transform of $U$ and $U^1,U^2$ denote the coordinates of $U$. Let $\hat{\cQ}$ denote the bilinear multiplier associated to $\cQ$, and define $\cQ(U,V,k,\s)$ to be
	\be
		\cQ(U,V,k,\s):=\sum_{\eta_1,\eta_2\in\ZZ}\hat{\cQ}(k\eta_1,k(\eta_2+\s))(\hat{U}(\eta_1),\hat{V}(\eta_2))e^{i(\eta_1+\eta_2)\xi}
	\ee
	A short computation involving the convolution theorem reveals that $e^{-i\s\xi}D_U(\tilde{u}_{\e,\kappa},k)e^{i\s\xi}V=2\cQ(\tilde{u}_{\e,\kappa},V,k,\s)$ for all $2\pi$-periodic $V$. This leads to the alternative description of the Bloch operator as
	\be
		B(\e,\kappa,\l,\s)=L(k,\mu;\s)+dk\d_\xi+iC(\e,\kappa)+2\cQ(\tilde{u}_{\e,\kappa},\cdot,k,\s)
	\ee
	Following the same procedure as the previous example, we need to simplify $PB_\s\Upsilon_\b$, $PB_{\s\s}\Upsilon_\b$, $PB_{\e\s}\Upsilon_\b$, $PB_\e\cV_\s$, $PB_\s\cV_\e$, and $PB_\s\cV_\s$. Of particular importance are the terms involving the nonlinearity.
	\begin{theorem}\label{theo:genreduced}
		The reduced equation matches the prediction of the complex Ginzburg-Landau model, in the sense that the reduced equation is given by \eqref{eq:GStableRedEqn}.
	\end{theorem}
	\begin{proof}
		For the terms only involving $\e$-derivatives, we appeal to Lemma \ref{lem:corefined}. Note that $B_\s(0,\kappa,\l,0)=k_*L_k(k_*,0;0)+iC(0,\kappa)$ and by design, we have that $PB_\s\Upsilon_\b=0$. We have that $B_{\s\s}(0,\kappa,\l,0)=k_*^2L_{kk}(k_*,0;0)$ and so a similar computation to the one in Lemma \ref{lem:GTaylorOSS} or \ref{lem:Example1OSS} gives the desired conclusion. We are now left with the $\cO(\e\s)$ term. As we are most interested in the terms involving $\cN$ and its derivatives, we will extract the relevant terms from $PB_{\e\s}\Upsilon_\b+PB_\s\cV_\e+PB_\e\cV_\s$ to get
		\be
			P\cQ_\s(\Upsilon_\a,\Upsilon_\b,k,0)-k_*PL_k(k_*,0)T_\l\cQ(\Upsilon_\a,\Upsilon_\b)-k_*P\cQ(\Upsilon_\a,T_\l L_k(k_*,0;0)\Upsilon_\b)
		\ee
		These terms are all zero because they are all Fourier supported in $\{0,\pm 2\}$ as $T_\l$, $L_k(k_*,0;0)$ are Fourier multiplier operators and $\cQ_\s(U,V,k,\s):=\frac{\d}{\d\s}\cQ(U,V,k,\s)$ is a bilinear Fourier multiplier. For the remaining terms, it is a similar computation to the one in Lemma \ref{lem:GTaylorOES} or \ref{lem:Example1OES} to show that they match the desired prediction.
	\end{proof}
	\subsection{General Quasilinear Nonlinearities}
	We now let $\cN(U,k,\mu)$ be a general quasilinear nonlinearity. To show that the reduced equation matches the prediction of complex Ginzburg-Landau, we follow the procedure outlined in the second local example. For the term in the Bloch operator coming from $D_U\cN(\tilde{u}_{\e,\kappa},k,\mu)$, we Taylor expand with respect to $U$ to get
	\ba
		e^{-i\s\xi}D_U\cN(\tilde{u}_{\e,\kappa},k,\mu)e^{i\s\xi}V&=e^{-i\s\xi}D_U^2\cN(0,k,\mu)(\tilde{u}_{\e,\kappa},e^{i\s\xi}V)+\\
		&\quad+\frac{1}{2}e^{-i\s\xi}D_U^3\cN(0,k,\mu)(\tilde{u}_{\e,\kappa},\tilde{u}_{\e,\kappa},e^{i\s\xi}V)+h.o.t.
	\ea
	The first main observation that we make is that the trilinear term doesn't contribute to the terms with $\s$-derivatives. This is because we need to take two $\e$-derivatives, one on each copy of $\tilde{u}_{\e,\kappa}$, and then taking a $\s$-derivative on top of that safely makes it an error term. Since we can handle the pure $\e$-derivatives using Lemma \ref{lem:corefined}, we can without essential loss of generality assume that $\cN(U,k,\mu)$ is a bilinear form. So from now on, assume that $N(U,k,\mu)=\frac{1}{2}D_U^2\cN(0,k,\mu)(U,U)$. Let $\cQ$ be the bilinear multiplier associated to $D_U^2\cN(0,k,\mu)$, that is
	\be
		D_U^2\cN(0,k,\mu)(U,V)(\xi)=\sum_{\eta_1,\eta_2\in\ZZ}\cQ(k\eta_1,k\eta_2,\mu)(\hat{U}(\eta_1),\hat{V}(\eta_2))e^{i(\eta_1+\eta_2)\xi}
	\ee
	For example $Q:\RR^n\times\RR^n\to\RR^n$ is a fixed bilinear form and $I,J\in\NN$, then the multiplier associated to $Q(\d_x^IU,\d_x^JV)$ is $\cQ(k\eta_1,k\eta_2)=(ik\eta_1)^I(ik\eta_2)^jQ$. As $D_U^2\cN(0,k,\mu)$ is a sum of forms of this type, we know that the multiplier $\cQ$ associated to $D_U^2\cN(0,k,\mu)$ is a smooth function $\cQ:\RR^2\to M_n(\CC)$. \\
	So the term coming from the nonlinearity can be computed as
	\be
		(e^{-i\s\xi}D_U\cN(\tilde{u}_{\e,\kappa},k,\mu)e^{i\s\xi})V(\xi)=\sum_{\eta_1,\eta_2\in\ZZ}\cQ(k\eta_1,k(\eta_2+\s),\mu)(\hat{U}(\eta_1),\hat{V}(\eta_2))e^{i(\eta_1+\eta_2)\xi}
	\ee
	As before, we define $D_U^2\cN(0,k,\mu,\s)$ to be the bilinear Fourier multiplier operator whose multiplier is $\cQ(k\eta_1,k(\eta_2+\s),\mu)$.
	\begin{theorem}
		The reduced equation matches the prediction of complex Ginzburg-Landau, in the sense that the reduced equation is given by \eqref{eq:GStableRedEqn}.
	\end{theorem}
	\begin{proof}
		It remains to compute the terms featuring a $\s$-derivative, which are $PB_\s\Upsilon_\b$, $PB_{\s\s}\Upsilon_\b$, $PB_{\e\s}\Upsilon_\b$, $PB_\s\cV_\s$, $PB_\e\cV_\s$, $PB_\s\cV_\e$. By design, we have that $PB_\s\Upsilon_\b=0$. For the term $PB_{\s\s}\Upsilon_\b+PB_\s\cV_\s$ we note that only derivatives of $L$ appear, and so the argument in Lemma \ref{lem:GTaylorOSS} and Lemma \ref{lem:Example1OSS} carry over. We are then left with $PB_{\e\s}\Upsilon_\b+PB_\s\cV_\e+PB_\e\cV_\s$ where the terms that explicitly depend on $D_U^2\cN(0,k,\mu)$ are given by
		\be
		P\d_\s D_U^2\cN(0,k,\mu,\s)(\Upsilon_\a,\Upsilon_\b,k,0)-k_*PL_k(k_*,0)T_\l\cQ(\Upsilon_\a,\Upsilon_\b)-k_*P\cQ(\Upsilon_\a,T_\l L_k(k_*,0;0)\Upsilon_\b)
		\ee
		But as in the model example, each of these terms are Fourier supported in $\{0,\pm 2\}$ and hence are automatically annihilated by $P$. This leaves the terms only depending on $L$, which can be handled by the same argument as in Lemma \ref{lem:GTaylorOES} or Lemma \ref{lem:Example1OES}.
	\end{proof}
	Once one has the reduced equation, the rest of the argument in Section \ref{sec:gencase} can be carried out giving the analog of Theorem \ref{thm:genstab}.
	\begin{theorem}\label{thm:genmatch}
		For any $\nu_0>0$, there exists an $\e_0>0$ so that if $\e<\e_0$ and $|\kappa|\leq \kappa_E$, $\tilde{u}_{\e,\kappa}$ is linearly stable if $\kappa^2\leq (1-\nu_0)\kappa_S^2$, and linearly unstable if $\kappa^2>(1+\nu_0)\kappa_S^2$, where $\kappa_E$ and $\kappa_S$ are as defined in Section \ref{sec:cglstability}. 
	\end{theorem}
	\begin{remark}\label{rem:nonlocal}
		For the nonlocal case, the first major obstacle is making sense of the Bloch operator. The main technical difficulty in this step is making sense of $e^{-i\s\xi}D_U\sN(\tilde{u}_{\e,\kappa},k,\mu)e^{i\s\xi}$. This is because a priori, $D_U\sN(U,k,\mu)$ can only act on exponentials whose frequencies are integers. To get around this, note that for all $(u,k)$ and all $q\in\NN$ we see that $(I_qu,\frac{k}{q})$ has the same image under the map $(u,k)\to u(kx)$. Hence, one has the following identity
		\begin{equation*}
			\sN(I_qu,\frac{k}{q},\mu)=I_{\frac{k}{q}}^{-1}\sN(I_{\frac{k}{q}}I_qu,\mu)=I_qI_k^{-1}\sN(I_ku,\mu)=I_q\sN(u,k,\mu)
		\end{equation*}
		By the chain rule, one then has
		\begin{equation*}
			I_q^{-1}D_U\sN(I_qu,\frac{k}{q},\mu)I_q=D_U\sN(u,k,\mu)
		\end{equation*}
		We can then define $D_U\sN(u,k,\mu)e^{i\frac{j}{q}\xi}$. for $j\in\ZZ$ to be
		\begin{equation*}
			D_U\sN(u,k,\mu)e^{i\frac{j}{q}\xi}:=I_q^{-1}D_U\sN(I_qu,\frac{k}{q},\mu)I_qe^{i\frac{j}{q}\xi}=I_q^{-1}D_U\sN(I_qu,\frac{k}{q},\mu)e^{ij\xi}
		\end{equation*}
		In the original lab frame coordinates, this identity comes from the observation that if $u$ is $\frac{2\pi}{k}$-periodic, then for any $q\in\NN$ $u$ is also $\frac{2\pi}{k/q}$-periodic as well and so by restricting $\sN$ to the subspace $H^s_{per}([0,\frac{2\pi}{k/q}];\RR^n)$ as opposed to $H^s_{per}([0,\frac{2\pi}{k}];\RR^n)$ allows us to extend $D_U\sN(u,\mu)$ from $\frac{2\pi}{k}$-periodic functions to $\frac{2\pi}{k/q}$-periodic functions. We will now check that $D_U\sN(u,k,\mu)e^{i\frac{j}{q}\xi}$ as defined above is well-defined. To do this, we start by assuming $\frac{j}{q}$ is such that $j,q$ are coprime, then any other $\frac{j'}{q'}=\frac{j}{q}$ is of the form $j'=nj$, $q'=nq$ for some $n\in\ZZ$. Then, we have that
		\begin{equation*}
			I_{q'}^{-1}D_U\sN(I_{q'}u,\frac{k}{q'},\mu)I_{q'}e^{i\frac{j'}{q'}\xi}=I_q^{-1}I_n^{-1}D_U\sN(I_nI_qU,\frac{k}{q'},\mu)I_nI_qe^{i\frac{j'}{q'}\xi}
		\end{equation*}
		By the chain rule computation above, we see that $I_n^{-1}D_U\sN(I_nI_qU,\frac{k}{q'},\mu)I_n=D_U\sN(I_qU,\frac{k}{q},\mu)$, and note that $I_qe^{i\frac{j'}{q'}\xi}=e^{i\frac{j'}{n}\xi}=e^{ij\xi}$, and so it is well-defined.\\
		Suppose that for some fixed $q$ that $\cK(\xi,\frac{j}{q};u,k,\mu):=e^{-i\frac{j}{q}\xi}D_U\sN(u,k,\mu)e^{i\frac{j}{q}\xi}$ is the corresponding Schwartz kernel. One important property that this object has in the local case is that $\cK$ is always $2\pi$-periodic in $\xi$ regardless of the frequency and our definition of $D_U\sN(u,k,\mu)e^{i\frac{j}{q}\xi}$ implies that $\cK$ can only be guaranteed to be $2\pi q$-periodic in the nonlocal case. The importance of this property is that ensures that the Bloch operator $B$ maps $2\pi$-periodic functions to $2\pi$-periodic functions for all $\s$. With this assumption in hand, and assuming that $\cK(\xi,\frac{j}{q};u,k,\mu)$ and the multiplier of $D_U^2\sN(0,k,\mu)$ admit smooth extensions to all frequencies, essentially the same argument as in the quasilinear case can be run to produce a linear stability result for these special nonlocal systems.
	\end{remark}
	
	%
	%



\end{document}